\documentclass[letterpaper,11pt,reqno]{amsart}
\usepackage[utf8]{inputenc}

\usepackage[
	showeqnr,
	theoremdefs,
	final
]{latexdev}
\usepackage[numbers]{natbib} 
\usepackage[all]{xy}
\usepackage{booktabs}
\usepackage{amssymb}
\usepackage{color}
\usepackage{amsxtra} 
\usepackage{tikz}

\usepackage{mathrsfs} 

\newcommand{\arxivorpnas}[2]{#1}

\usepackage{amsmath,amssymb}

\DeclareMathOperator{\grad}{grad}
\DeclareMathOperator{\divv}{div}

\newcommand{\ee}{\mathrm{e}}
\newcommand{\ii}{\mathrm{i}}
\newcommand{\pair}[1]{\left\langle #1 \right\rangle}

\newcommand{\inner}[1]{\langle\!\langle #1 \rangle\!\rangle}
\providecommand{\norm}[1]{\lVert#1\rVert}
\providecommand{\abs}[1]{\lvert#1\rvert}

\newcommand{\ud}{\mathrm{d}}

\newcommand{\pd}{\partial}

\newcommand{\RR}{{\mathbb R}}
\newcommand{\CC}{{\mathbb C}}

\newcommand{\vol}{\mu}
\newcommand{\met}{\mathsf{g}}
\newcommand{\Diff}{\mathrm{Diff}}
\newcommand{\Xcal}{\mathfrak{X}}
\newcommand{\Diffvol}{{\Diff_\vol}}

\newcommand{\LieD}{\mathcal{L}}

\newcommand{\Dens}{\mathrm{Dens}}

\DeclareMathOperator{\Ad}{Ad}

\newcommand*\SU{\mathrm{SU}}

\newcommand{\Met}{\mathsf{G}}
\newcommand{\MetW}{ {\bar{\mathsf{G}}} }
\newcommand{\MetF}{ {{\mathsf{G}}} }

\renewcommand{\Im}{\mathrm{Im}}
\renewcommand{\Re}{\mathrm{Re}}

\newcommand{\marginnote}[1]
{
}

\newcounter{gm}

\newcounter{bk}
\newcommand{\bk}[1]
{\stepcounter{bk}$^{\bf BK\thebk}$%
\footnotetext{\hspace{-3.7mm}$^{\blacksquare\!\blacksquare}$
{\bf BK\thebk:~}#1}}

\newcounter{km}

\title[Geometry of the Madelung transform]{Geometry of the Madelung transform}

\author{Boris Khesin}
\address{Department of Mathematics, University of Toronto, Toronto, ON M5S 2E4, Canada}
\email{khesin@math.toronto.edu}

\author{Gerard \texorpdfstring{Misio\l ek}{Misiolek}}
\address{Department of Mathematics, University of Notre Dame, Notre Dame, IN 46556, USA}
\email{gmisiole@nd.edu}

\author{Klas Modin}
\address{Department of Mathematical Sciences, Chalmers University of Technology and University of Gothenburg, SE-412 96 Gothenburg, Sweden}
\email{klas.modin@chalmers.se}

\begin{document} 
\begin{abstract} 
The Madelung transform is known 
to relate Schr\"odinger-type equations in quantum mechanics and the Euler equations for 
baro\-tropic-type fluids. We prove that, more generally, the Madelung transform is a K\"ahler map 
(i.e.\ a symplectomorphism and an isometry) between the space of wave functions and the cotangent bundle to the density space  equipped with the Fubini-Study metric and  the Fisher-Rao information metric, respectively.
We also show that Fusca's momentum map property  of the Madelung transform is a manifestation of the general 
approach via reduction for semi-direct product groups. 
Furthermore, the Hasimoto transform for the binormal equation turns out to be 
the 1D case of the Madelung transform, while its higher-dimensional version is related to 
the problem of conservation of the Willmore energy in binormal flows.  
\end{abstract} 

\maketitle 

\tableofcontents 

\section{Introduction} 
\label{sec:introduction}

In 1927 E.\ Madelung~\cite{Ma1927} introduced a transformation, which now bears his name, in order to give 
an alternative formulation of the linear Schr\"odinger equation for a single particle 
moving in an electric field as a system of equations describing the motion of 
a compressible inviscid fluid. 
Since then other derivations have been proposed in the physics literature primarily in connection 
with various models in quantum hydrodynamics and optimal transport, 
cf.\ \cite{Ne2011, Re2012, Mo2015b}. 

In this paper we focus on the geometric aspects of Madelung's construction and prove that 
the Madelung transform possesses a number of surprising properties. 
It turns out that in the right setting it can be viewed as a symplectomorphism, an isometry, a K\"ahler morphism 
or a generalized Hasimoto map. 
Furthermore, geometric properties of the Madelung transform are best understood not in the setting of 
the $L^2$-Wasserstein geometry but (an infinite-dimensional analogue of) 
the Fisher-Rao information geometry---the canonical Riemannian geometry of 
the space of probability densities. 
These results can be summarized in the following theorem (a joint version of Theorems \ref{thm:madelung_symplectomorphism} and \ref{thm:madelung_isometry} below). 

\medskip

\noindent{\bf Main Theorem.} {\it 
The Madelung transform is a K\"ahler morphism between the cotangent bundle of the space of 
smooth probability densities, equipped with the (Sasaki)-Fisher-Rao metric, and an open subset 
of the infinite-dimensional complex projective space of smooth wave functions, 
equipped with the Fubini-Study metric.}
%
\medskip

The statement is valid in both the Sobolev topology of $H^s$-smooth functions and Fr\'echet topology of $C^\infty$-smooth functions.
In a sense the Madelung transform resembles the passage from Euclidean to polar coordinates 
in the infinite-dimensional space of wave functions, where the modulus is a probability
density and the phase corresponds to fluid's vector field. 
The above theorem shows that, after projectivization, this transform relates not only 
equations of hydrodynamics and those of quantum physics, 
but the corresponding symplectic structures underlying them as well. 
Surprisingly, it also turns out to be an isometry between two well-known Riemannian metrics 
in geometry and statistics.

This result reveals tighter links between hydrodynamics, quantum information geometry 
and geometric quantum mechanics. 
Important in our constructions is a reformulation of Newton's equations on these spaces of diffeomorphisms 
and probability densities. This reformulation can be viewed as an extension of Arnold's formalism 
for the Euler equations of ideal hydrodynamics 
\cite{Ar1966, ArKh1998}. 


\smallskip

Our first motivation comes from hydrodynamics where 
groups of diffeomorphisms arise as configuration spaces for flows of compressible and incompressible fluids 
in a domain $M$ (typically, a compact connected Riemannian manifold with a volume form $\vol$).
When equipped with a metric given at the identity diffeomorphism by the $L^2$ inner product 
(corresponding essentially to the kinetic energy) the geodesics of the group $\Diff(M)$ 
of smooth diffeomorphisms of $M$ describe motions of the gas of noninteracting particles in $M$ 
whose velocity field $v$ satisfies the inviscid Burgers equation 
$$ 
\dot v+\nabla_v v=0. 
$$ 
On the other hand, when restricted to the subgroup $\Diff_\vol(M)$ of volume-preserving diffeomorphisms,
the $L^2$-metric becomes right-invariant, and its geodesics can be viewed as motions of 
an ideal (that is, incompressible and inviscid) fluid in $M$ whose velocity field satisfies 
the incompressible Euler equations 
$$ 
\begin{cases} 
\begin{aligned} 
&\dot v+\nabla_v v=-\nabla p 
\\ 
&{\rm div}\, v=0. 
\end{aligned} 
\end{cases} 
$$ 
Here the pressure gradient $\nabla p$ is defined uniquely  by 
the divergence-free condition on the velocity field $v$ and can be viewed as a constraining force 
on the fluid. 
What we describe below can be regarded as an extension of this framework 
to various equations of compressible fluids, where the evolution of density becomes foremost important.

\smallskip

Our second motivation is to study the geometry of the space of densities. 
Namely, consider the projection $\pi \colon \Diff(M) \to \Dens(M)$ of the full diffeomorphism group $\Diff(M)$
onto the space $\Dens(M)$  of normalized smooth densities on $M$. 
The fiber over  a density $\nu$ consists of all diffeomorphisms $\phi$ that push forward 
the Riemannian volume form $\vol$ to $\nu$, that is, $\phi_*\vol=\nu$. 
It was shown by Otto \cite{Ot2001} that $\pi$ is a Riemannian submersion between $\Diff(M)$ 
equipped with the $L^2$-metric and $\Dens(M)$ equipped with the (Kantorovich-Wasserstein) metric 
used in the optimal mass transport. 
More interesting for our purposes is that a Riemannian submersion arises also when 
$\Diff(M)$ is equipped with a right-invariant homogeneous Sobolev $\dot{H}^1$-metric 
and $\Dens(M)$ with the Fisher-Rao metric which plays an important role in geometric statistics, see~\cite{KhLeMiPr2013}. 

 \smallskip

In the present paper we  prove  the K\"ahler property of the Madelung transform thus establishing 
a close relation of  the cotangent space of the space of densities and the projective space 
of wave functions on $M$. Furthermore, this transform 
also identifies many Newton-type equations on these spaces that are naturally related to 
equations of fluid dynamics. 

\smallskip

As an additional perspective, the connection between equations of quantum mechanics and hydrodynamics described below might shed some light on the hydrodynamical quantum analogs studied in \cite{Co_et_al_2005, Bush2010}: the motion of bouncing droplets in certain vibrating liquids manifests
many properties of quantum mechanical particles.
While bouncing droplets have a dynamical boundary condition with changing topology of the domain every period, apparently a more precise description of the phenomenon should involve a certain averaging procedure for the hydrodynamical system in a periodically changing domain.
Then the droplet--quantum particle correspondence could be a combination of
the averaging and Madelung transform.

\medskip

\textbf{Acknowledgements.} B.K. is grateful to the IHES in Bures-sur-Yvette  and the
Weizmann Institute in Rehovot for their support and kind
hospitality. B.K. was also partially supported by an NSERC research
grant. Part of
this work was done while G.M. held the Ulam Chair visiting
Professorship in University of Colorado at Boulder.
K.M. was supported by EU Horizon 2020 grant No 691070, by the Swedish Foundation for International Cooperation in Research and Higher Eduction (STINT) grant No PT2014-5823, and by the
Swedish Research Council (VR) grant No 2017-05040.

\section{Madelung transform as a symplectomorphism} 
\label{sec:madelung}

In this section we show that the Madelung transform induces a symplectomorphism between 
the cotangent bundle of smooth probability densities and the projective space of smooth 
non-vanishing complex-valued wave  functions.
\begin{definition}
Let $\vol$ be a (reference) volume form on $M$ such that $\int_M \vol = 1$.
The space of probability densities on a compact connected oriented $n$-manifold $M$ is 
\begin{equation}
	\Dens^s(M) = \Big\{ \rho\in H^s(M)\mid \rho > 0, \; \int_M \rho\,\vol = 1 \Big\},
\end{equation} 
where $H^s(M)$ denotes the space of real-valued functions on $M$ of Sobolev class $H^s$ with $s > n/2$ 
(including the case $s=\infty$ corresponding to $C^\infty$ functions).\footnote{From 
a geometric point of view it is more natural to define densities as volume forms instead of functions.
	This way, they become independent of the reference volume form $\vol$.
	However, since some of the equations studied in this paper depends on the reference volume form $\vol$ anyway, it is easier to define densities as functions to avoid notational overload.}

%
\end{definition}

The space $\Dens^s(M)$ can be equipped in the standard manner with the structure of a smooth infinite-dimensional 
manifold (Hilbert, if $s<\infty$ or Fr\'echet, if $s=\infty$).
It is an open subset of an affine hyperplane in $H^s(M)$. 
Its tangent bundle is trivial 
\begin{equation*}
T\Dens^s(M)=\Dens^s(M)\times H^s_0(M) 
\end{equation*} 
where 
$H^s_0(M) = \big\{ c\in H^s(M) \mid \int_M c\,\vol = 0 \big\}$. 
Likewise, the (regular part of the) co-tangent bundle is 
\begin{equation*} 
T^*\Dens^s(M)=\Dens^s(M)\times H^s(M)/\RR ,
\end{equation*} 
where $H^s(M)/\RR$ is the space of cosets $[\theta]$ of functions $\theta$ modulo additive constants 
$[\theta]=\{\theta+c~|~c\in \RR\}$. 
The pairing is given by 
\begin{equation*}
	T_\rho\Dens^s(M)\times T_\rho^*\Dens^s(M) \ni (\dot\rho,[\theta]) \mapsto \int_M \theta\dot\rho\,\vol. 
\end{equation*}
It is independent of the choice of $\theta$ in the coset $[\theta]$ since $\int_M\dot\rho \,\vol=0$.

%
%
%
\begin{definition}\label{def:madelung} 
The \textit{Madelung transform} is a map ${\bf \Phi}$ which to any pair of functions 
$\rho\colon M\to \mathbb{R}_{>0}$ and $\theta\colon M\to \RR$ 
associates a complex-valued function 
\begin{equation}\label{eq:madelung_def}
	{\bf\Phi}\colon (\rho,\theta)\mapsto \psi \coloneqq \sqrt{\rho \ee^{\ii\theta}} = \sqrt{\rho}\,\ee^{\ii\theta/2} \, .
\end{equation}	
\end{definition} 
\begin{remark} 
The latter expression defines a particular branch 
of the square root $\sqrt{\rho \ee^{\ii\theta}}$. 
The map ${\bf\Phi}$ is unramified, since $\rho$ is strictly positive.
Note that this map is not injective because $\theta$ and $\theta+4\pi k $ have the same image. 
Despite this fact, 
there is, as we shall see next, a natural geometric setting in which the Madelung transform  \eqref{eq:madelung_def} becomes invertible. 
\end{remark} 

\begin{figure}
	\includegraphics[width=0.8\textwidth]{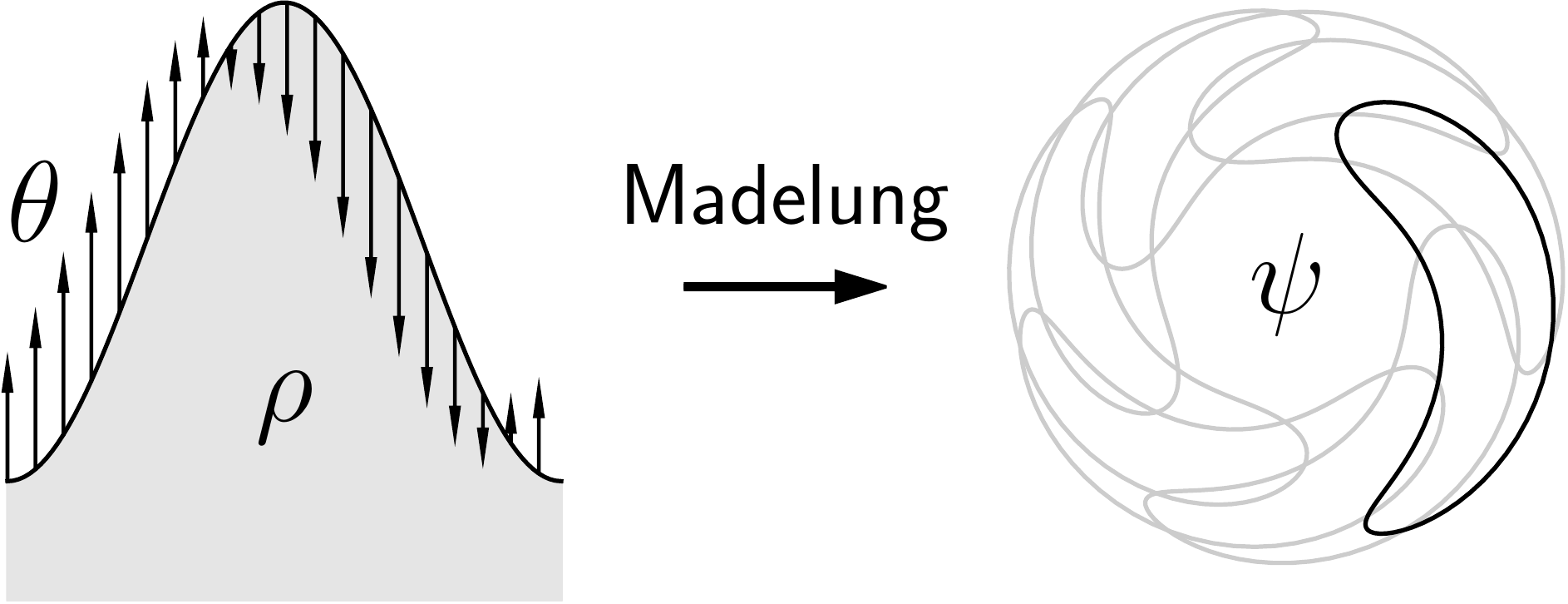}
	\caption{Illustration of the Madelung transform $\bf\Phi$ on $S^1$. For $x\in S^1$, a probability density $\rho(x)>0$ and a dual infinitesimal probability density $\theta(x)$ are mapped to a wave function $\psi(x)=\sqrt{\rho \ee^{\ii \theta}}\in\CC$, which is defined up to rigid rotations of the complex plane.}
\end{figure}

\subsection{Symplectic properties} 
Let $H^s(M,\CC)$ denote the space of complex-valued functions of Sobolev class on a compact connected manifold $M$
and let $\mathbb PH^s(M,\CC)$ denote the corresponding complex projective space. 
Its elements can be represented as cosets of the unit $L^2$-sphere of complex functions 
$$ 
[\psi]=\Big\{\ee^{\ii \tau}\psi ~|~\psi\in  H^s(M,\CC),\, \norm{\psi}_{L^2}=1 ~\text{and~} \tau\in\RR \Big\}. 
$$ 

If $\tilde{\psi}\in [\psi]$ is nowhere vanishing then every other representative in the coset $[\psi]$
is nowhere vanishing as well. 
In particular, $\mathbb PH^s(M,\CC\backslash \{0\})$ is an open subset and hence a submanifold of 
$\mathbb PH^s(M,\CC)$. 
\begin{theorem}\label{thm:madelung_symplectomorphism}
The Madelung transform \eqref{eq:madelung_def} induces a map 
\begin{equation}\label{eq:madelung_geometric} 
{\bf\Phi}\colon T^*\Dens^s(M)\to \mathbb PH^{s}(M,\CC\backslash \{0\}) 
\end{equation} 
%
which, up to scaling by $4$, is a symplectomorphism\footnote{In the Fr\'echet topology of smooth functions if $s=\infty$.} 
with respect to the canonical symplectic structure of $T^*\Dens^s(M)$ 
and the symplectic form of the K\"ahler structure on $\mathbb PH^s(M,\CC)$.
\end{theorem}
\begin{proof}
We need to establish the following three steps: 
(i) ${\bf\Phi}$ is well-defined, (ii) ${\bf\Phi}$ is smooth, surjective and injective and (iii) ${\bf\Phi}$ is symplectic.

(i) Let $\rho\in\Dens^s(M)$. 
Recall that 
the elements of $T^{*}_{\rho}\Dens^s(M)$ are cosets of $H^s$ functions on $M$ modulo constants 
and given any $\theta \in H^s(M,\mathbb{R})$ and any $\tau \in \mathbb{R}$ 
the Madelung transform maps $(\rho,\theta+\tau)$ to $\sqrt{\rho}\ee^{\ii(\theta+\tau)/2}$. 
If $s>n/2$ then standard results on products and compositions of Sobolev functions 
(cf. e.g., \cite{Pa1968})
show that it is smooth as a map to $H^{s}(M,\CC)$. 
Furthermore, we have 
$$
\big\| \sqrt{\rho}\ee^{\ii(\theta+\tau)/2} \big\|_{L^2} 
= 
\norm{\sqrt{\rho}\ee^{\ii\theta/2}}_{L^2} = \norm{\sqrt{\rho}}_{L^2} = 1
$$ 
so that that cosets $({\rho}, [\theta])$ are mapped to cosets $[\psi]$, i.e.,  the map is well-defined.

\smallskip

(ii) Surjectivity and smoothness of ${\bf\Phi}$ are evident. To prove injectivity for the cosets 
recall that inverting the Madelung map amounts essentially to rewriting of a non-vanishing 
complex-valued function in polar coordinates. Since preimages for a given $\psi$ 
differ by a constant polar argument 
$\tilde\theta=\theta+2\pi k$, they define the same coset $[\theta]$. 
Similarly, changing $\psi$ by a constant phase does not affect the argument coset $[\theta]$, 
which implies injectivity of the map between the cosets 
$({\rho}, [\theta])$ and $[\psi]$.\footnote{Note that the injectivity would not hold for $L^2$ functions, 
or even for smooth functions if $M$ were not connected. Indeed, the arguments of the preimages 
could then have incompatible integer jumps at different points of $M$. 
For continuous functions on a connected $M$ it suffices to fix the argument at one point only.} 
	
\smallskip	
(iii) The canonical symplectic form on $T^*\Dens^s(M)$ is given by 
\begin{equation}\label{eq:scaled_can_sympl} 
\Omega^{T^*\Dens}_{(\rho,[\theta])}\big((\dot\rho_1,[\dot\theta_1]),(\dot\rho_2,[\dot\theta_2])\big) \vol
= 
\int_M \left( \dot\theta_1\dot\rho_2 - \dot\theta_2\dot\rho_1\right)\vol. 
\end{equation} 
Since $\int_M \dot\rho_k\,\vol = 0$ it follows that it is well-defined on the cosets $[\dot\theta_i]$. 
The symplectic form on $\mathbb PH^s(M,\CC)$ is given by 
\begin{equation} \label{eq:wPH} 
\Omega^{\mathbb PH^s}_{[\psi]}\big([\dot\psi_1],[\dot\psi_2]\big) 
= 
\int_M \Im\big(\dot\psi_1\overline{\dot\psi_2} \big)\vol. 
\end{equation} 
The tangent vectors can be described as cosets 
$[\dot\psi_k] = \{ \ii c \psi + \dot\psi_k \mid c\in \RR \}$ 
obtained by differentiating 
$ [\psi]= \{ \psi\ee^{\ii\tau}\mid \tau\in \RR \}$.
One can see that $\Omega^{\mathbb PH^s}$ is well-defined on the coset vectors which follows from
 $\int_M \Re \big( \psi \dot {\overline{\psi}}_k \big) \vol = 0$ and a straightforward calculation.
Finally, for $\psi = {\bf\Phi}(\rho,[\theta])$ the tangent vector is 
$T_{(\rho,[\theta])}{\bf\Phi}(\dot\rho,[\dot\theta]) 
=1/2 ( \dot\rho/\rho + \ii\dot\theta ){\bf\Phi}(\rho,[\theta]) $. 
%
Then \eqref{eq:wPH} gives 
\begin{align*} 
&\Omega^{\mathbb PH^s}_{{\bf\Phi}(\rho,[\theta])}\left(T_{(\rho,[\theta])}{\bf\Phi}(\dot\rho_1, [\dot\theta_1]),T_{(\rho,[\theta])}{\bf\Phi}(\dot\rho_2,[\dot\theta_2])\right) 
= \\ &= 
\frac{1}{4}\int_M \Im\left( 
\Big( \frac{\dot\rho_1}{\rho} + \ii \dot\theta_1 \Big)
\Big( \frac{\dot\rho_2}{\rho} - \ii \dot\theta_2 \Big) 
\psi\overline{\psi}  \right)\vol 
= \\ &= 
\frac{1}{4}\int_M \left(\dot\theta_1\frac{\dot\rho_2}{\rho} 
-  
\dot\theta_2\frac{\dot\rho_1}{\rho} \right)\rho\,\vol
= 
\frac{1}{4}\int_M \left( \dot\theta_1\dot\rho_2 - \dot\theta_2\dot\rho_1 \right) \vol
\\ &= 
\frac{1}{4}\Omega^{T^*\Dens}_{(\rho,[\theta])} 
\big((\dot\rho_1,[\dot\theta_1]),(\dot\rho_2,[\dot\theta_2])\big)\,, 
\end{align*} 
which completes the proof. 
\end{proof} 
\begin{remark}
In \autoref{sec:momentum} the inverse Madelung transform is defined for any $C^1$ function 
with no restriction on strict positivity of $\abs{\psi}^2$. 
It can be defined similarly in a Sobolev setting. 
Furthermore, extending \citet{Fu2017}, we will also show  that it can be understood
as a momentum map for a natural action of a certain semi-direct product group. 
Thus the Madelung transform relates 
the standard symplectic structure on the space of wave functions and the linear Lie-Poisson 
structure on the corresponding dual Lie algebra.
\end{remark} 
\begin{remark}
The fact that the Madelung transform is a symplectic submersion between the cotangent bundle 
of the space of densities and the unit sphere $S^\infty\subset H^s(M,\CC\backslash \{0\})$ of non-vanishing wave functions was proved by \citet{Re2012}. 
The stronger symplectomorphism property proved in \autoref{thm:madelung_symplectomorphism} 
is achieved by considering projectivization $\mathbb PH^s(M,\CC\backslash \{0\})$. 
\end{remark}
%

\subsection{Example: linear and nonlinear Schr\"odinger equations} \label{sub:nl_schrodinger} 
Let $\psi$ be a wave function on a Riemannian manifold $M$ and consider the family of Schr\"odinger 
(or Gross-Pitaevsky) equations with Planck's constant $\hbar=1$ and mass $m=1/2$ 
\begin{equation}\label{eq:schrodinger} 
 	\mathrm{i}\dot\psi = - \Delta\psi +  V\psi + f(\abs{\psi}^2)\psi, 
\end{equation} 
where $V\colon M\to \RR$ and $f\colon (0,\infty) \to \RR$. 
If $f\equiv 0$ we obtain the linear Schr\"odinger equation with potential~$V$. 
If $V\equiv 0$ we obtain the family of non-linear Schr\"odinger equations (NLS); 
two typical choices are $f(a) = \kappa a$ 
and $f(a) = \frac 12(a-1)^2$. 

Note that Equation \eqref{eq:schrodinger} is Hamiltonian with respect to the symplectic structure 
induced by the complex structure of $L^2(M,\CC)$.
Indeed, recall that the real part of a Hermitian inner product defines a Riemannian structure 
and the imaginary part defines a symplectic structure, so that 
\begin{equation*} 
\Omega(\psi_1,\psi_2): = \Im \inner{ \psi_1,\psi_2 }_{L^2} 
= 
\Re \inner{ \mathrm{i}\psi_1,\psi_2 }_{L^2} 
\end{equation*} 
defines a symplectic form $\Omega$ corresponding to the complex structure $J(\psi)=\ii\psi$. 
The Hamiltonian function for the Schrödinger equation~ \eqref{eq:schrodinger} is
\begin{equation} \label{eq:ham_NLS_psi} 
H(\psi) 
= 
\frac{1}{2}\norm{\nabla\psi}_{L^2}^{2} + \frac{1}{2}\int_M \left( V \abs{\psi}^2 + F(\abs{\psi}^{2}) \right)\vol, 
\end{equation} 
where $F\colon (0,\infty) \to \RR$ is a primitive function of $f$, namely $F'=f$. 

Observe that the $L^2$-norm of any solution $\psi$ of \eqref{eq:schrodinger} is conserved in time. 
Furthermore, the Schr\"odinger equation is also equivariant with respect to a constant phase shift 
$\psi(x)\mapsto e^{\ii\tau}\psi(x)$ 
and therefore descends to the projective space $\mathbb PH^s(M,\CC)$. 
It can be viewed as an equation on the complex projective space, 
a point of view first suggested in \cite{Ki1979}. 

\begin{proposition}[cf.\ \cite{Ma1927,Re2012}] 
The Madelung transform \eqref{eq:madelung_geometric} maps the family of Schr\"odinger equations \eqref{eq:schrodinger} to the following system on $T^*\Dens^s(M)$ 
\begin{equation}\label{eq:ham_eq} 
\left\{ 
\begin{aligned}
&\dot\theta +  \frac{1}{2}\abs{\nabla\theta}^2 - \frac{4 \Delta\sqrt{\rho}}{\sqrt{\rho}} + 2V + 2f(\rho) = 0,
\\
&\dot\rho + \divv(\rho {\nabla\theta}) = 0.
\end{aligned}\right. 
\end{equation} 
Equation \eqref{eq:ham_eq} has a hydrodynamic formulation as an equation for 
a barotropic-type fluid  
\begin{equation}\label{eq:barotropic2} 
\left\{ 
\begin{aligned} 
&\dot v + \nabla_v v + \nabla\Big(2V + 2f(\rho) - \frac{4 \Delta\sqrt{\rho}}{\sqrt{\rho}} \Big) = 0 
\\ 
&\dot\rho +\divv(\rho v) = 0 
\end{aligned} \right. 
\end{equation} 
with potential velocity field $v=\nabla \theta$. 
\end{proposition}
\begin{remark}
Note  that \eqref{eq:ham_eq} only makes sense for $\rho>0$, whereas the NLS equation 
makes sense even when $\rho\geq 0$. In particular, the properties of the Madelung transform 
imply that if one starts with a wave function such that $\abs{\psi}^2 >0$ everywhere, 
then it remains strictly positive for all $t$ for which the solution to equation~\eqref{eq:ham_eq} is defined, since this holds for $\rho=\abs{\psi}^2$ by the continuity equation.
Thus, $\abs{\psi}^2$ can become non-positive only if $v=\nabla\theta$ stops being a $C^1$ vector field (so that the continuity equation breaks).
\end{remark}
%
%


\begin{proof}
Since the transformation $(\rho,[\theta])\mapsto \psi$ is symplectic, it is enough to work out the Hamiltonian~\eqref{eq:ham_NLS_psi} expressed in $(\rho,[\theta])$.
First, notice that 
\begin{equation}\label{eq:grad_psi}
	\nabla \psi = \ee^{\ii\theta/2} \Big( \nabla\sqrt{\rho} + \frac{\ii}{2}\sqrt{\rho}\,\nabla \theta \Big),
\end{equation}
so that 
\begin{equation}\label{eq:norm_psi}
\begin{split}
\norm{\nabla\psi}_{L^2}^{2} 
&= 
\inner{ \nabla \sqrt{\rho}+\frac{\ii}{2} \sqrt{\rho}\nabla\theta,\nabla \sqrt{\rho} 
+ 
\frac{\ii}{2} \sqrt{\rho}\nabla\theta }_{L^2} 
\\
&= 
\inner{ \nabla\sqrt{\rho},\nabla\sqrt{\rho}}_{L^2} + 
\frac{1}{4}\inner{\rho \nabla\theta,\nabla\theta }_{L^2} 
\end{split}
\end{equation}
Thus, the Hamiltonian on $T^*\Dens^s(M)$ corresponding to 
the Schr\"odinger Hamiltonian \eqref{eq:ham_NLS_psi} is 
\begin{equation*}
H(\rho,[\theta]) 
= 
\frac{1}{2}\int_M \left(\frac{1}{4}\abs{\nabla\theta}^2\rho + \abs{\nabla\sqrt\rho}^2 \right)\vol 
+ 
\frac{1}{2}\int_M \left( V\rho+ F(\rho)\right)\vol. 
\end{equation*}
Since
\begin{equation*}
\frac{\delta H}{\delta \rho} = \frac{1}{8}\abs{\nabla\theta}^2 - \frac{\Delta\sqrt\rho}{\sqrt\rho} 
+ 
\frac{1}{2}V + \frac{1}{2} f(\rho)
\quad \text{and} \quad 
\frac{\delta H}{\delta \theta} = -\frac{1}{4}\divv(\rho\nabla\theta) 
\end{equation*}
the result now follows from Hamilton's equations 
\begin{equation*}
\dot\theta = -4\frac{\delta H}{\delta \rho}, \quad \dot\rho = 4\frac{\delta H}{\delta\theta}
\end{equation*}
for the canonical symplectic form \eqref{eq:scaled_can_sympl} scaled by $1/4$.
\end{proof}
\begin{corollary}
The Hamiltonian system (\ref{eq:ham_eq}) on $T^*\Dens^s(M)$ for potential solutions of 
the barotropic equation \eqref{eq:barotropic2} is mapped symplectomorphically to 
the Schr\"odinger equation (\ref{eq:schrodinger}). 
\end{corollary}
\begin{remark}
Conversely, classical PDE of hydrodynamic type can be expressed as NLS-type equations. 
For example, potential solutions $v=\nabla \theta$ of the compressible Euler equations of a barotropic fluid
are Hamiltonian on $T^*\Dens^s(M)$ with the Hamiltonian given as the sum of the kinetic energy 
$K=\frac 12\int_M |\nabla \theta|^2\rho\,\vol$
and the potential energy $U=\int_M e(\rho)\,\rho\,\vol$, where $e(\rho)$ is the fluid internal energy, 
see \cite{KhMiMo2018}. 
They can be formulated as an NLS equation with the Hamiltonian
\begin{equation}\label{eq:compressible_Euler_NLS_Hamiltonian}
H(\psi) 
= 
\frac{1}{2}\norm{\nabla\psi}_{L^2}^{2} 
- 
\frac{1}{2}\norm{\nabla \abs{\psi}}_{L^2}^{2} 
+ 
\int_M e(\abs{\psi}^{2})\abs{\psi}^{2}\vol. 
\end{equation} 
The choice $e=0$ gives a Schr\"odinger formulation for potential solutions of Burgers' equation, 
which describe geodesics in the $L^2$-type Wasserstein metric on $\Dens^s(M)$.
Thus, the geometric framework links the optimal transport for cost functions with potentials with 
the compressible Euler equations and the NLS-type equations described above.
\end{remark} 
\medskip

\subsection{Madelung transform as a Hasimoto map in 1D} \label{sec:hasimoto}

The celebrated {\it vortex  filament   equation} 
$$
\dot\gamma=\gamma'\times \gamma''
$$
is an evolution equation on a (closed or open) curve 
$\gamma\subset \RR^3$, where  $\gamma=\gamma(x,t)$ and  $\gamma':=\partial \gamma/\partial x$
and $x$ is an arc-length parameter along $\gamma$. 
(An equivalent {\it binormal form} of this  equation $\dot \gamma=k(x,t)\bf b$ 
is valid in any parametrization, 
where $\bf b=t\times n$ is the binormal unit vector to the curve at a point $x$, 
$\bf t$ and $\bf n$ are, respectively, the unit tangent and the normal vectors 
and $k(x,t)$ is the curvature of the curve at the point $x$ at moment $t$). 
This equation describes a localized induction approximation of the 3D Euler equation 
of an ideal fluid in $\RR^3$, where the vorticity of the initial velocity field is supported on a curve $\gamma$. 
(Note that the corresponding evolution of the vorticity 
is given by the hydrodynamical Euler equation, which becomes nonlocal in terms of vorticity. By considering the ansatz that keeps only local terms, it reduces to the filament  equation above.)
\smallskip

The vortex filament equation is known to be Hamiltonian with respect to 
the Marsden-Weinstein symplectic structure on the space of curves in $\RR^3$ and with Hamiltonian 
given by the length functional, see, e.g., \cite{ArKh1998}.
\begin{definition}
The   Marsden-Weinstein symplectic structure $\Omega^{MW}$
assigns to a pair of two variations $V,W$ of a curve $\gamma$ 
(understood as vector fields on $\gamma\subset \RR^3$) 
the value 
$\Omega^{MW} (V,W):=\int_\gamma i_V i_W\vol$, 
where $\vol$ is the Euclidean volume form in $\RR^3$. 
\end{definition}
It turns out that the vortex filament equation becomes the equation of the 1D barotropic-type fluid \eqref{eq:barotropic2} 
with $\rho=k^2$ and $v=2\tau$, where $k$ and $\tau$ denote curvature and torsion of the curve $\gamma$, respectively.


In 1972 Hasimoto~\cite{Ha1972} introduced the following surprising transformation.
\begin{definition} \label{def:hasimoto} 
The {\it Hasimoto transformation} assigns 
to a curve $\gamma$, with curvature $k$ and torsion $\tau$, a wave function $\psi$ 
according to the formula 
$$ 
(k(x),\tau(x))\mapsto \psi(x)=k(x)e^{\ii\int_{x_0}^x\tau(\tilde x)\ud\tilde x}. 
$$ 
\end{definition} 
This map takes the vortex filament  equation to the 1D NLS equation 
$\ii\dot \psi+\psi''+\frac 12 |\psi|^2\psi=0\,.$ 
(A change of the initial point $x_0$ in $\int_{x_0}^x \tau(\tilde x)d\tilde x$ 
leads to a multiplication of $\psi(x)$ by an irrelevant  constant phase $e^{\ii\alpha}$). 
In particular, the filament equation becomes a completely integrable system 
whose first integrals are obtained by pulling back those of the NLS equation. 
The first integrals for the filament equation can be written in terms of 
the total length $\int dx$, the torsion $\int \tau\,\ud x$, the squared curvature $\int k^2\,\ud x$, 
followed by $ \int \tau k^2\,\ud x$ etc. 
%
%


\begin{figure}
	\centering
	\begin{tikzpicture}
		\node[anchor=south west, inner sep=0] (image) at (0,0) {\includegraphics[width=0.5\textwidth]{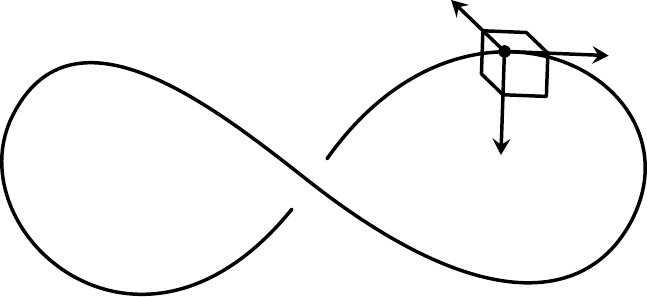}};
		\begin{scope}[x={(image.south east)},y={(image.north west)}]
			\coordinate (gdot) at (0.69,0.99) {};
			\coordinate (gprime) at (0.93,0.76) {};
			\coordinate (gbis) at (0.73,0.52) {};
			\coordinate (g) at (0.45,0.2) {};
			\node[left, rotate=0] at (gdot) {$\dot\gamma = \gamma'\times\gamma''$};
			\node[above right, rotate=0] at (gprime) {$\gamma'$};
			\node[below right, rotate=0] at (gbis) {$\gamma''$};
			\node[right, rotate=0] at (g) {$\gamma$};
		\end{scope}
	\end{tikzpicture}
	\caption{Vortex filament flow: each point of the curve $\gamma$ moves in the direction of the
	binormal. If $k(x)$ and $\tau(x)$ are the curvature and torsion at $\gamma(x)$, then the wave
	function  $\psi(x) = k(x) \ee^{\ii \int_{x_0}^x \tau(\tilde x)\ud \tilde x}$ 
		satisfies the NLS equation. Moreover, the pair
	of functions $v=2\tau$ and $\rho=k^2$ satisfies the equation of the 1D barotropic fluid, which
	is a manifestation of the 1D Madelung transform.}\label{fig:binormal}
\end{figure}
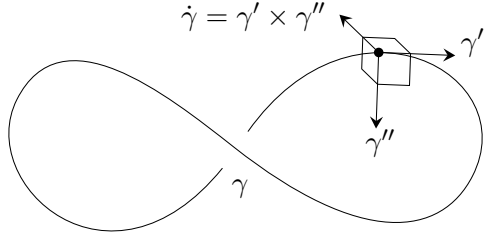


\begin{remark} \label{rem:LP}
Each of the three forms of the above equations has a natural  symplectic or Poisson structure: 
the Marsden-Weinstein  symplectic structure on nonparametrized curves $\{\gamma\}$ for
the binormal equation, the linear Lie-Poisson structure on (the dual of) the semidirect product 
$\mathfrak s= \Xcal(S^1)\ltimes H^s(S^1)\ni (v,\rho)$
for the 1D compressible Euler equation on $v=2\tau$ and $\rho=k^2$, 
and the standard constant symplectic structure on wave functions $\{\psi\}$ for the NLS.

Langer and Perline \cite{LaPe1991} established symplectic properties of the Hasimoto transform. It turns out that the 
Marsden-Weinstein symplectic structure expressed in terms of the curvature and torsion is mapped by the Hasimoto 
transform to the constant symplectic structure on wave functions. (The original statement in \cite{LaPe1991} is more complicated, since
the passage from the curve $\gamma$ to its curvature and torsion requires taking two extra derivatives.)
This symplectic property has the following heuristic explanation.
The Marsden-Weinstein symplectic structure $\Omega^{MW}$ on curves in $\RR^3$ is, essentially, averaging of the 
standard symplectic structures in all normal planes to the curve $\gamma$. Furthermore, one can regard the 
curvature magnitude $|k(x)|$ as the radial coordinate in each normal plane, while  (twice) the integral of torsion 
$\theta(x)=2\int \tau(x)\,\ud x$ as the angular coordinate (since torsion is by definition the angular velocity of the rotation of the normal vector). 
This means that the passage from affine coordinates in normal planes to the polar ones $(k^2, \theta/2)$ is a symplectic map: $dx\wedge dy=(1/2)d(k^2)\wedge d\theta$.
On the other hand, $k$ and $\theta$ are (adjusted) polar coordinates of the wave function $\psi(x)=k(x)e^{\ii\theta(x)}$.
So one arrives at the standard symplectic structure on the wave functions, regarded as complex-valued functions.
\end{remark}
The following proposition relates the Hasimoto transform to the classical Madelung transform, 
see Section \ref{sec:madelung}. 
\begin{proposition} 
The Hasimoto transformation is the Madelung transform in the 1D case.
\end{proposition}
This can be seen by comparing Definitions \ref{def:madelung} and \ref{def:hasimoto} which make the Hasimoto transform seem much less surprising. 
Alternatively, one may note that for $\psi(x)=\sqrt{\rho(x)}e^{\ii\theta(x)/2}$ the pair $(\rho, v)$ 
with $v=\nabla\theta$ satisfies the compressible Euler equation, while in the one-dimensional case 
these variables are expressed via the curvature $\sqrt{\rho}=\sqrt{k^2}=k$ 
and the (indefinite) integral of torsion $\theta(x)/2=\int v( x)d  x=\int \tau( x)d  x$. 
\begin{remark} 
The filament equation  has a higher-dimensional analog for membranes 
(i.e., compact oriented surfaces $\Sigma$ of co-dimension 2 in $\RR^n$) 
as a skew-mean-curvature flow
$ 
\dot q={\bf J}({\bf MC}(q)), 
$ 
where $q\in \Sigma$ is any point of the membrane, ${\bf MC}(q)$ is the mean curvature vector to $\Sigma$ 
at the point $q$ and $\bf J$ is the operator of rotation by $\pi/2$ in the positive direction in every normal space 
to $\Sigma$. 
This equation is again Hamiltonian with respect to the Marsden-Weinstein structure $\Omega^{MW}$
on membranes of co-dimension 2 
and with a Hamiltonian function given by the $(n-2)$-dimensional volume of the membrane, see e.g.\ \cite{Sh2012}.
\end{remark}
An intriguing problem in this area is the following.

\begin{question} Find an analogue of the Hasimoto map, which sends a 
skew-mean-curvature flow to an NLS-type equation for any $n$.
\end{question}


The existence of the Madelung transform and its symplectic property in any dimension 
is a strong indication that such an analog should exist. 
Indeed, in any dimension by means of the Madelung transform one can pass from the wave function 
evolved according to an NLS-type equation to the polar form of $\psi$, i.e. to its magnitude $\sqrt{\rho}$ 
and the phase $\theta$, so that the pair $(\rho, v)$ with $v=\nabla\theta$ will evolve according to 
the compressible Euler equation. 
Thus for a surface $\Sigma$ of co-dimension 2 moving according to
the skew-mean-curvature flow, the problem boils down to interpreting the corresponding characteristics 
$(\rho, \nabla\theta)$ similarly to the one-dimensional curvature and torsion.
(Note that both the pair $(\rho, \theta)$ and the co-dimension 2 surface $\Sigma$ in $\RR^n$
can be encoded by two functions of $n-2$ variables). 

In any dimension the square of the mean curvature vector can be regarded as a natural
analog of the density $\rho=\|{\bf MC}\|^2$. In this case an analog of the total mass of the fluid, 
i.e. 
$\int_\Sigma\rho\,\ud \sigma$, is the {\it Willmore energy} 
$\mathcal{W}(\Sigma)=\int_\Sigma \|{\bf MC}\|^2\,\ud \sigma$.
An intermediate step in finding a higher-dimensional Hasimoto map is then the following 


\begin{conjecture}
For a co-dimension 2 surface $\Sigma\in \RR^n$ moving by the skew-mean curvature flow 
$\dot q={\bf J}({\bf MC}(q))$ the following equivalent properties hold:

{\rm i)} its Willmore energy $\mathcal{W}(\Sigma)$ is invariant, 

{\rm ii)}
 its square mean curvature $\rho=\|{\bf MC}\|^2$ evolves according to the continuity equation $\dot \rho+{\rm div}(\rho v)=0$ for some vector field $v$ on $\Sigma$.
\end{conjecture}


The equivalence of the two statements is 
a consequence of Moser's theorem: 
if the total mass on a surface is preserved, the corresponding evolution of density can be realized 
as a flow of a time-dependent vector field. 
\begin{proposition}
The conjecture is true in dimension 1. 
\end{proposition} 
\begin{proof}
In 1D the conservation of the Willmore energy is the time invariance of the integral
$\mathcal{W}(\gamma)=\int_\gamma k^2 \,\ud x$ or, equivalently, in the arc-length parameterization, 
of the integral 
$\int_\gamma |\gamma''|^2\,\ud x$. 
The latter invariance follows from the following straightforward computation 
\begin{align*} 
\frac{d}{dt}\mathcal{W}(\gamma) 
&= 
2\int_\gamma (\dot\gamma'',\gamma'')\,\ud x 
= 
-2\int_\gamma (\dot\gamma',\gamma''')\,\ud x 
\\ 
&= 
-2\int_\gamma ((\gamma'\times \gamma'')',\gamma''')\,\ud x=0. 
\end{align*} 
\end{proof} 
It would be very interesting to find a higher-dimensional analog of the torsion $\tau$ 
for co-dimension 2 membranes.
Note that the integral of the torsion has to play the role of an angular coordinate 
in the tangent spaces to $\Sigma$. The torsion would be the gradient part of the field $v$ 
transporting the density $\rho=\|{\bf MC}\|^2$. 
Essentially, the question is how to encode a co-dimension 2 surface by its mean curvature and torsion. 
Presumably $\nabla\theta$ as an analog of $\tau$ can be regarded as an angle of rotation  (``the phase") of 
the vector $\bf MC$, i.e.  it might play a role of an exact 1-form. 

\begin{question}
Is such a surface $\Sigma \subset \RR^n$ of co-dimension 2 reconstructable (modulo isometries) 
from the vectors $\bf MC$, i.e. from their magnitude $\|{\bf MC}\|$ and an ``angle of rotation", 
an exact 1-form $d\theta$?
\end{question}

Finally, note that such a higher-dimensional Hasimoto map should inherit the Poisson properties 
of the Madelung transform. 
The heuristic argument of Remark \ref{rem:LP} concerning the relation of the symplectic structure 
in the two-dimensional normal bundle and the space of wave functions should work in any dimension. 
The Madelung transform between complex-valued wave functions and pairs consisting of densities 
and gradient potentials has been already shown to be symplectic, see Section \ref{def:madelung}.

\section{Madelung transform as an isometry of K\"ahler manifolds} 
\label{sub:kahler_properties_of_madelung}

\subsection{Metric properties} 
In this section we prove that the Madelung transform is an isometry and a K\"ahler map 
between the lifted Fisher-Rao metric on the cotangent bundle $T^{*}\Dens^s(M)$ and the K\"ahler structure 
corresponding to the Fubini-Study metric on the infinite-dimensional projective space $\mathbb PH^s(M,\CC)$.

\begin{definition}\label{def:fisher_rao}
The Fisher-Rao metric on the density space $\Dens^s(M)$ is given by 
\begin{equation}\label{eq:fisher_rao} 
\MetF_{\rho}(\dot\rho,\dot\rho) 
= 
\frac{1}{4}\int_M \frac{\dot\rho^2}{\rho}\,\vol. 
\end{equation} 
\end{definition} 
This metric is invariant under the action of the diffeomorphism group. 
It is, in fact, the \emph{only} Riemannian metric on $\Dens^s(M)$ with this property, 
cf. e.g., \cite{BaBrMi2016}.

Next, observe that an element of $TT^{*}\Dens^s(M)$ is a 4-tuple $(\rho,[\theta],\dot\rho,\dot\theta)$, 
where $\rho\in\Dens^s(M)$, $[\theta] \in H^s(M)/\RR$, $\dot\rho \in H^s_0(M)$ 
and $\dot\theta \in H^s(M)$ subject to the constraint 
\begin{equation} 
\int_M \dot\theta \rho\,\vol = 0. 
\end{equation} 
\begin{definition} 
The lift of the Fisher-Rao metric to the cotangent bundle $T^{*}\Dens^s(M)$ has the form 
\begin{equation}\label{eq:sasaki_FR_metric} 
\MetF^*_{(\rho,[\theta])}\big( (\dot\rho,\dot\theta),(\dot\rho,\dot\theta) \big) 
= 
\frac{1}{4}\int_{M} \bigg( \frac{\dot\rho^2}{\rho} 
+ 
\dot\theta^2\rho \bigg) \vol. 
\end{equation} 
We will refer to this metric as the \emph{Sasaki-Fisher-Rao} metric.
\end{definition}

Next, recall that the canonical (weak) \emph{Fubini-Study metric} on the complex projective space 
$\mathbb PH^s(M,\CC) \subset \mathbb PL^2(M,\CC)$ is given by
\newcommand{\FS}{\mathsf{FS}} 
\begin{equation}\label{eq:fubini_study} 
\FS_\psi(\dot\psi,\dot\psi) 
= 
\frac{\inner{\dot\psi,\dot\psi}_{L^2}}{\inner{\psi,\psi}_{L^2}} 
- 
\frac{\inner{\psi,\dot\psi}_{L^2} \inner{\dot\psi,\psi}_{L^2}}{\inner{\psi,\psi}_{L^2}^{2}}. 
\end{equation} 
\begin{theorem}\label{thm:madelung_isometry} 
The Madelung transform ${\bf\Phi}\colon T^{*}\Dens^s(M)\to \mathbb PH^s(M,\CC)$ is an isometry 
with respect to the Sasaki-Fisher-Rao metric~\eqref{eq:sasaki_FR_metric} on $T^{*}\Dens^s(M)$ 
and the Fubini-Study metric~\eqref{eq:fubini_study} on $\mathbb P H^s(M,\CC\backslash \{0\})$.
\end{theorem} 
\begin{proof}
We have 
\begin{equation*} 
T_{(\rho,[\theta])}{\bf\Phi} (\dot\rho,\dot\theta) 
= 
\frac{\dot\rho}{2\sqrt{\rho}} \, e^{\ii\theta/2} + \frac{\ii\dot\theta \sqrt{\rho}}{2} \, e^{\ii\theta/2} 
= 
\frac{1}{2}\left( \frac{\dot\rho}{\rho} + \ii \dot\theta\right)\psi\,,
\end{equation*} 
where $\psi = {\bf\Phi}(\rho,[\theta])$. Since $\|\psi\|_{L^2}^2 = 1$, setting 
$\dot\psi = T_{(\rho,\theta)}{\bf\Phi} (\dot\rho,\dot\theta)$ 
we obtain 
\begin{equation*} 
\FS_\psi(\dot{\psi},\dot{\psi}) 
= 
\inner{\dot\psi,\dot\psi}_{L^2} - \inner{\dot\psi,\psi}_{L^2} \inner{\psi,\dot\psi}_{L^2} \,,
\end{equation*} 
where 
\begin{equation*} 
\inner{\dot\psi,\dot\psi}_{L^2} 
= 
\frac{1}{4}\int_{M} \left| \frac{\dot\rho}{\rho} + \ii \theta \right|^{2} \rho\,\vol 
= 
\frac{1}{4}\int_M \bigg( \frac{\dot\rho^{2}}{\rho^{2}} + \dot\theta^{2} \bigg) \rho 
= 
\MetF^{*}_{(\rho,\theta)}(\dot\rho,\dot\theta) 
\end{equation*} 
and 
\begin{equation*} 
\inner{\dot\psi,\psi}_{L^2} 
= 
\frac{1}{2} \int_{M} \bigg( \frac{\dot\rho}{\rho} + \ii\dot\theta  \bigg) \rho \,\vol 
= 
\frac{1}{2} \int_M \dot\rho \,\vol+ \frac{\ii}{2} \int_M \dot\theta\rho\,\vol = 0 \,,
\end{equation*} 
which proves the theorem. 
\end{proof}

The metric property in \autoref{thm:madelung_isometry} combined with the symplectic property 
in \autoref{thm:madelung_symplectomorphism} yields the following. 
\begin{corollary} 
The cotangent bundle $T^*\Dens^s(M)$ is a Kähler manifold with the Sasaki-Fisher-Rao metric  \eqref{eq:sasaki_FR_metric} and the canonical symplectic structure \eqref{eq:scaled_can_sympl} 
scaled by $1/4$. 
The corresponding integrable almost complex structure is given by
\begin{equation}\label{eq:complex_structure} 
J_{(\rho,[\theta])}(\dot\rho,\dot\theta) = \bigg(\dot\theta\rho,-\frac{\dot\rho}{\rho} \bigg). 
\end{equation} 
\end{corollary}
This result can be compared with the result of Molitor \cite{Mo2015b} who described 
a similar construction using (the cotangent lift of) the $L^2$ Wasserstein metric in optimal transport 
but obtained an almost complex structure on $T^\ast\Dens^s(M)$ which is not integrable. 
It appears that the Fisher-Rao metric is a more natural choice for such constructions: its lift to 
$T^\ast\Dens^s(M)$ admits a compatible complex (and K\"ahler) structure. 
It would be interesting to write down K\"ahler potentials for all metrics compatible 
with \eqref{eq:complex_structure} 
and identify which of these are invariant under the action of the diffeomorphism group.


\subsection{Geodesics of the Sasaki-Fisher-Rao metric}

As an isometry the Madelung transform maps geodesics of the Sasaki metric to geodesics 
of the Fubini-Study metric. 
The latter
are projective lines in the projective space of wave functions. 
To see which submanifolds are mapped to projective lines by the Madelung transform 
we need to describe geodesics of the Sasaki-Fisher-Rao metric.
\begin{proposition}\label{prop:geodesic_eq_SFR}
Geodesics of the Sasaki-Fisher-Rao metric~\eqref{eq:sasaki_FR_metric} on the cotangent bundle 
$T^{*}\Dens^s(M)$ satisfy the system 
\begin{equation*}
		\left\{
		\begin{aligned}
		\frac{\ud}{\ud t}\left( \frac{\dot\rho}{\rho} \right) &= - \frac{1}{2}\left(\frac{\dot\rho}{\rho}\right)^{2} + \frac{\dot\theta^{2}}{2}\,, \\
		\frac{\ud}{\ud t}\left( \dot\theta \rho \right) &= 0\,.			
		\end{aligned}
		\right.
\end{equation*} 
\end{proposition}
\begin{proof} 
The Lagrangian is given by the metric 
$L(\rho,\theta,\dot\rho,\dot\theta ) = \MetF^{*}_{(\rho,\theta)}((\dot\rho,\dot\theta),(\dot\rho,\dot\theta))$. 
The variational derivatives are obtained from the formulas 
\begin{align*}
		\frac{\delta L}{\delta \dot\rho} &= \frac{1}{2}\frac{\dot\rho}{\rho}\,, & \frac{\delta L}{\delta \dot\theta } &= \frac{1}{2}\dot\theta\rho\,, \\
		\frac{\delta L}{\delta \rho} &= - \frac{1}{4}\left(\frac{\dot\rho}{\rho}\right)^2 + \frac{1}{4}\dot\theta^{2}\,, & \frac{\delta L}{\delta\theta} &= 0 \,,
\end{align*} 
which yield the equations of motion as stated. 
\end{proof}
\begin{remark} 
The natural projection $(\rho,[\theta]) \mapsto \rho$ is a Riemannian submersion between 
$T^{*}\Dens^s(M)$ equipped with the Sasaki-Fisher-Rao metric \eqref{eq:sasaki_FR_metric} 
and $\Dens^s(M)$ equipped with the Fisher-Rao metric \eqref{eq:fisher_rao}. 
The corresponding horizontal distribution on $T^*\Dens^s(M)$ is given by 
\begin{equation*} 
\mathrm{Hor}_{(\rho,[\theta])} 
= 
\big\{ (\dot\rho,\dot\theta)\in T_{(\rho,[\theta])}\Dens^s(M) \mid \dot\theta = 0 \big\}. 
\end{equation*} 
Indeed, if $\dot\theta = 0$ then the equations of motion of \autoref{prop:geodesic_eq_SFR}, 
restricted to $(\rho,\dot\rho)$, yield the geodesic equations for the Fisher-Rao metric. 
One can think of this as a zero-momentum symplectic reduction corresponding to 
the abelian gauge symmetry 
$(\rho,[\theta]) \mapsto (\rho,[\theta + f])$ for any function $f\in H^s(M)$. 
\end{remark} 
\subsection{Example: 2-component Hunter-Saxton equation}

This is a system of two equations 
\begin{equation}\label{eq:two_HS_eq}
\left\{
\begin{array}{l}
		\dot v'' = -2 v' v'' - v v''' + \sigma\sigma'\,, \\
		\dot\sigma = - (\sigma v)' \,,
\end{array} \right. 
\end{equation} 
where $v=v(t,x)$ and $\sigma=\sigma(t,x)$ are time-dependent periodic functions on the line 
and the prime stands for the $x$-derivative.
It can be viewed as a high-frequency limit of the 2-component Camassa-Holm equation, 
cf.\ \cite{HaWu2011}. 

It turns out that this system is closely related to the K\"ahler geometry of the Madelung transform 
and the Sasaki-Fisher-Rao metric~\eqref{eq:sasaki_FR_metric}. 
Consider the semi-direct product $\mathcal{G} = \Diff_0^{s+1}(S^1)\ltimes H^s(S^1,S^1)$, 
where $\Diff_0^{s+1}(S^1)$ is the group of circle diffeomorphisms that fix a prescribed point 
and $H^s(S^1,S^1)$ is the space of Sobolev $S^1$-valued maps of the circle. 
The group multiplication is given by 
\begin{equation*} 
(\varphi,\alpha)\cdot(\eta,\beta) = (\varphi\circ\eta,\beta + \alpha\circ\eta). 
\end{equation*} 
Define a right-invariant Riemannian metric on $\mathcal G$ at the identity element by 
\begin{equation} \label{eq:lenells_metric} 
\inner{ (v,\sigma), (v,\sigma) }_{\dot{H}^1} 
= 
\frac{1}{4} \int_{S^1} \left( (v')^2 +  \sigma^2\right) \ud x. 
\end{equation} 
If $t \to (\varphi(t), \alpha(t))$ is a geodesic in $\mathcal{G}$ then 
$v = \dot\varphi\circ\varphi^{-1}$ and $\sigma = \dot\alpha\circ\varphi^{-1}$ 
satisfy equations \eqref{eq:two_HS_eq}.
Lenells \cite{Le2013b} showed that the map 
\begin{equation}\label{eq:lenells_map} 
(\varphi,\alpha) \mapsto \sqrt{\varphi' \, \mathrm{e}^{\ii\alpha}} 
\end{equation} 
is an isometry from $\mathcal G$ to a subset of 
$\{ \psi \in H^s(S^1,\CC) \mid \norm{\psi}_{L^2} = 1 \}$. 
Moreover, solutions to \eqref{eq:two_HS_eq} satisfying $\int_{S^1} \sigma  \, \ud x = 0$ 
correspond to geodesics in the complex projective space $\mathbb PH^s(S^1,\CC)$ 
equipped with the Fubini-Study metric. 
Our results show that this isometry is a particular case of \autoref{thm:madelung_isometry}. 
%
\begin{proposition}\label{prop:2HS_as_Sasaki}
The 2-component Hunter-Saxton \arxivorpnas{equation}{} \eqref{eq:two_HS_eq} with initial data satisfying
$\int_{S^1} \sigma \,\ud x = 0$ is equivalent to the geodesic equation of the Sasaki-Fisher-Rao metric 
\eqref{eq:sasaki_FR_metric} on $T^\ast\Dens^s(S^1)$. 
\end{proposition} 
\begin{proof}
First, observe that the mapping \eqref{eq:lenells_map} can be rewritten as 
$(\varphi,\alpha) \mapsto {\bf\Phi}(\pi(\varphi),\alpha)$, 
where ${\bf\Phi}$ is the Madelung transform and $\pi$ is the projection 
$\varphi\mapsto \varphi^*\vol$ specialized to the case $M=S^1$. 

Next, observe that the metric \eqref{eq:lenells_metric} in the case $\int_{S^1}\sigma \ud x = 0$ 
is the pullback of the Sasaki metric \eqref{eq:sasaki_FR_metric} by the mapping 
\begin{equation*} 
\Diff_0^{s+1}(S^1)\ltimes H^s(S^1,S^1) 
\ni 
(\varphi,\alpha) \mapsto (\pi(\varphi),[\theta]) \in T^*\Dens^s(S^1), 
\end{equation*} 
where $\theta(x) = \int_0^x \alpha'(s)\ud s$. 
Indeed, we have 
\begin{align*} 
\MetF^*_{(\pi(\varphi),[\theta])}\bigg( \frac{\ud}{\ud t}\pi(\varphi), [\dot\alpha] \bigg) 
&= 
\frac{1}{4}\int_{S^1} \bigg( \Big( \frac{\dot\varphi'}{\varphi'} \Big)^2 
+ 
\dot\alpha^2 \bigg) \varphi' \,\ud x 
\\ &= 
\frac{1}{4}\int_{S^1} \left( \left( \pd_x(\dot\varphi\circ\varphi^{-1}) \right)^2 
+ 
(\dot\alpha\circ\varphi^{-1})^2 \right) \ud x 
\\ &= 
\frac{1}{4}\int_{S^1} \left(  (v')^2 + \sigma^2 \right) \ud x. 
\end{align*} 

It follows from the change of variables formula by the diffeomorphism $\varphi$ that the condition 
$\int_{S^1}\sigma\ud x = 0$ corresponds to $\int_{S^1}\dot\alpha \varphi'\ud x = 0$. 
Hence, the description of the 2-component Hunter-Saxton equation as a geodesic equation 
on the complex projective $L^2$ space is a special case of that on $T^*\Dens^s(M)$ 
with respect to the Sasaki-Fisher-Rao metric \eqref{eq:sasaki_FR_metric}. 
\end{proof}

\begin{remark}
Observe that if $\sigma=0$ at $t=0$ then $\sigma(t)=0$ for all $t$ and 
the 2-component Hunter-Saxton \arxivorpnas{equation}{} \eqref{eq:two_HS_eq} 
reduces to the standard Hunter-Saxton equation. 
This is a consequence of the fact that horizontal geodesics on $T^*\Dens^s(M)$ 
with respect to the Sasaki-Fisher-Rao metric descend to geodesics on $\Dens^s(M)$ 
with respect to the Fisher-Rao metric.
\end{remark} 
%

\section{Madelung transform as a momentum map} 
\label{sec:momentum} 

In \autoref{sec:madelung}
we described the Madelung transform as a symplectomorphism 
from $T^*\Dens(M)$ to $\mathbb PH^s(M,\CC\backslash\{0\})$
which associates a wave function $\psi=\sqrt{\rho}e^{\ii\theta/2}$ (modulo a phase factor $e^{\ii \tau}$) 
to a pair $(\rho, [\theta])$ consisting of a density $\rho$ of unit mass 
and a function $\theta$ (modulo an additive constant). 
Here, we start by outlining (following \cite{Fu2017}) another approach, which shows that 
it is natural to regard the inverse Madelung transform as a momentum map 
from the space $\mathbb PH^s(M,\CC)$ of wave functions $\psi$ to the set of pairs 
$(\rho\,\ud\theta,\rho)$ regarded as elements of the dual space $\mathfrak s^*$ of a certain Lie algebra. 
The latter is a semidirect product $\mathfrak s=\Xcal(M)\ltimes H^s(M)$ corresponding to 
the Lie group $S=\Diff(M)\ltimes H^s(M)$. 
(In this section $\Diff(M)$ stands for Sobolev diffeomorphisms $\Diff^{s+1}(M)$
and $\Xcal(M)$ for vector fields $\Xcal^{s+1}(M)$.)

Furthermore, this construction generalizes to the vector-valued case 
$S_{(\ell)}=\Diff(M)\ltimes H^s(M,\mathbb{C}^\ell)$. 
For $\ell=2$ this group appears naturally in the description of general compressible fluids 
including transport of both density and entropy. 
The case $\ell>1$ provides also a setting for quantum systems with spin degrees of freedom. 
For example, $\ell=2$ (a rank-1 spinor) describes fermions with spin $1/2$ 
(such as electrons, neutrons, and protons). 

In \autoref{sect:general-semi} below we present a unifying point of view which explains the origin 
of the Madelung transform as the momentum map in a semidirect product reduction. 


\begin{subsection}{A group action on the space of wave functions}
We start by defining a group action on the space of wave functions. 
First, observe that it is natural to think of $H^s(M,\CC)$ as a space of complex-valued 
half-densities on $M$. Indeed, $\psi\in H^s(M,\CC)$ is assumed to be square-integrable 
and $|\psi|^2$ is interpreted as a probability measure.
Half-densities are characterized by how they are transformed under diffeomorphisms 
of the underlying space: the pushforward $\varphi_{*}\psi$ of a half-density $\psi$ on $M$ 
by a diffeomorphism $\varphi$ of $M$ is given by the formula 
$$ 
\varphi_*\psi = \sqrt{|\mathrm{Det}(D\varphi^{-1})|}\, \psi \circ \varphi^{-1}. 
$$ 

This formula explains the following natural action of a semidirect product group on the vector space 
of half-densities.
\begin{definition}\cite{Fu2017}\label{def:action}
	The semidirect product group $S=\mathrm{Diff}(M)\ltimes H^s(M)$ {\it acts on the  space} 
	$H^s(M,\CC)$  as follows: 
	for a group element $(\varphi,a)\in S$ the action on wave functions $\psi$ is
	\begin{equation}\label{eq:grpact}
			(\varphi,a) \circ \psi = \sqrt{|\mathrm{Det}(D\varphi^{-1})|}\, e^{-\ii a/2} (\psi \circ \varphi^{-1}).
\end{equation} 
This action descends to the space of cosets $[\psi]\in \mathbb P H^s(M,\CC)$.
\end{definition}

Thus, a wave function $\psi$ is pushed forward under the diffeomorphism $\varphi$ as a complex-valued half-density, followed by a pointwise phase adjustment by $e^{-\ii a/2}$. 
An easy computation gives the following Lie algebra action 
\begin{proposition} 
The infinitesimal action of $\xi=(v,\alpha) \in \mathfrak{s}$ on $H^s(M,\CC)$ corresponding to 
the action \eqref{eq:grpact} is represented by the vector field $V_\xi$ on $H^s(M,\CC)$ 
defined at each point $\psi$ by
\begin{equation}\label{eq:infact} 
V_\xi(\psi) = -\frac{1}{2} \psi \divv(v)  - \frac{\ii}{2} \alpha \psi - \iota_v\ud \psi. 
\end{equation} 
\end{proposition} 
%
\end{subsection}

\begin{subsection}{The inverse of the Madelung transform} 

Consider the following alternative definition of the inverse  Madelung transform, 
which will be our primary object here. 
Let $\Omega^1(M)$ denote the space of 1-forms on $M$ of Sobolev class $H^s$. 
Recall the definition~\eqref{eq:madelung_def} of the Madelung transform: 
$(\rho,\theta) \mapsto \psi = \sqrt{\rho e^{i\theta}}$, 
where $\rho>0$.

\begin{proposition}{\rm \cite{Fu2017}}
The map  
\begin{align}\label{eq:mad}
{\bf M}\colon &H^s(M,\CC) \to \Omega^1(M)\times\Dens^s(M) 
\end{align} 
given by 
$$ 
\psi\mapsto(m, \rho) = \big( 2\,\Im ( \bar \psi\,\ud \psi) ,\bar \psi \psi \big) 
$$ 
is the inverse of the Madelung transform~\eqref{eq:madelung_def} in the following sense:
if $\psi = \sqrt{\rho e^{\ii\theta}}$ then ${\bf M}(\psi) = (\rho\,\ud\theta , \rho)$. 
\end{proposition}
\begin{proof} For $\psi=\sqrt{\rho}e^{\ii\theta/2}$ one evidently has $\bar \psi \psi=\rho$.
The expression for the other component follows from the observation 
$$ 
\Im \, \bar \psi \,\ud \psi 
=
\bar \psi \psi\, \Im \,  \ud\,( \ln \psi) = \rho\, \Im \,  \ud\, ((\ln \sqrt{\rho}) + {\ii\theta/2}) 
= 
\rho \,\ud \theta/2. 
$$
These two components allow one to obtain $\rho$ and $\rho \,\ud\theta$ and hence, by integration, 
to recover $\theta$ modulo an additive constant. 
(The ambiguity involving an additive constant in the definition of $\theta$ corresponds to 
recovering the wave function $\psi$ modulo a constant phase factor). 
\end{proof}
For a positive function $\rho$ satisfying $\int_M\rho \,\vol = 1$
the pair $(\rho\,\ud\theta, \rho)$ can be identified with $(\rho,[\theta])$ in $T^*\Dens(M)$, 
where the momentum variable $m=\rho\,\ud\theta$ 
is naturally thought of as an element of $\Xcal(M)^*$. 
Note, however, that this definition of the inverse Madelung works in greater generality: 
the momentum variable $m$ is defined even when $\rho$ is allowed to be zero, 
although $\theta$ cannot be recovered there. 
\begin{remark} 
So far we viewed $\psi$ as a function on an $n$-manifold $M$.
One can also consistently regard  $\psi$ as a complex half-density $\varpi = \psi\,\vol^{1/2}$.
The set of complex half-densities on $M$ is denoted $\sqrt{\Omega^{n}}(M)\otimes \CC$ 
indicating that it is ``the square root" of the space  $\Omega^n(M)$ of $n$-forms. 
%
%
Then  the map ${\bf M}$ in \eqref{eq:mad}
can be understood as follows. 
For a half-density $\varpi\in\sqrt{\Omega^{n}}(M)\otimes \CC$ the second component $\bar{\varpi} \varpi$ 
of the map $\bf M$ is understood as a tensor product $(\bar\psi\psi)\,\vol=\rho\,\vol$ 
of two half-densities on $M$, thus yielding the density $\rho\in \Dens^s(M)$. 
One can show that the first component $ \Im \, (\bar \varpi \,\ud \varpi )$ of ${\bf M}$ 
can be regarded as an element $m\otimes \vol=\rho\,\ud\theta\otimes \vol\in
\Omega^1(M)\otimes_{H^s(M)} \Omega^n(M)$.  Namely, given a reference density 
$\vol$, for any half-density $\varpi=f(x)\vol^{1/2}$
define its differential $\ud\varpi := \ud f(x)\otimes  \vol^{1/2}$. While the differential $\ud\varpi$ depends 
on the choice of the reference density, the momentum map does not. 
\end{remark}
\begin{proposition}
For any half density $\varpi=f(x)\vol^{1/2}\in \sqrt{\Omega^{n}}(M)\otimes \CC$ 
the momentum
$2 \Im \, (\bar \varpi \ud \varpi  ) =  2\Im \,\bar f \ud f\otimes  \vol$
is a well defined element of $\Omega^1(M)\otimes_{H^s(M)} \Omega^n(M)$
and does not depend on the choice of the reference density $\vol$.
\end{proposition}
\begin{proof} 
Given a different reference volume form $\nu=h(x) \vol$ with a positive function $h>0$ 
one has $\varpi=f(x)\vol^{1/2}=g(x)\nu^{1/2}=g(x)(h(x)\vol)^{1/2}$, where $f(x)=g(x)\sqrt{h(x)}$ 
and 
\begin{align*} 
\Im \, (\bar \varpi \ud \varpi  ) 
&=  
\Im \,\bar f \ud f \otimes \vol 
= 
\Im \,\bar g \sqrt{h}\,  d(g \sqrt{h}) \otimes \vol 
\\ 
&= 
\Im \,(\bar g \sqrt{h} \sqrt{h} \,\ud g + \bar g  \sqrt{h}  g \,\ud (\sqrt{h}))\otimes  \vol 
\\ 
&= 
\Im \,\bar g h \, \ud g \otimes \vol = \Im \,\bar g \,\ud g \otimes_{H^s(M)} (h \vol) 
\\ 
&= 
\Im \,\bar g  \ud g \otimes  \nu, 
\end{align*} 
where we dropped the term with $\bar g  g \sqrt{h}   \,\ud (\sqrt{h})$ since it is purely real.
\end{proof} 
\begin{remark} 
The pair $(m, \rho)\otimes \vol=( \rho\,\ud \theta\otimes \vol, \rho\, \vol)$ 
is understood as an element of the space
$\mathfrak s^*=\Omega^1(M)\otimes_{H^s(M)} \Omega^n(M) \oplus \Omega^n(M)$ 
dual to the Lie algebra $\mathfrak s = \Xcal(M)\ltimes H^s(M)$, 
while the inverse Madelung transformation is a map ${\bf M} \colon H^s(M,\CC) \to \mathfrak s^*$. 
Note that the dual space $\mathfrak s^*$ has a natural Lie-Poisson structure 
(as any dual Lie algebra).\end{remark}
\end{subsection} 

\begin{subsection}{A reminder on momentum maps}
In the next section we  show that the inverse Madelung transform \eqref{eq:mad} is a momentum map 
associated with the action \eqref{eq:grpact} of the Lie group $S=\mathrm{Diff}(M)\ltimes H^s(M)$ 
on $H^s(M,\CC)$. We start by recalling the definition of a momentum map. 

Suppose that a Lie algebra $\mathfrak g$ acts on a Poisson manifold $P$ 
and denote its action by $A\colon\mathfrak g \to \mathfrak{X}(P)$ where $A(\xi)=\xi_P$. 
Let $\langle \, , \rangle$ denote the pairing of $\mathfrak g$ and $\mathfrak g^*$. 
\begin{definition}\label{def:mommap}
A \textit{momentum map} associated with a Lie algebra action $A(\xi)=\xi_P$ is a map 
${\mathcal M}\colon P \to \mathfrak g^*$ such that for every $\xi \in \mathfrak g$ 
the function $H_{\xi} \colon P \to \RR$ defined by 
$H_{\xi}(p) \coloneqq \langle {\mathcal M}(p), \xi \rangle$ 
for any $p\in P$ is a Hamiltonian of the vector field $\xi_P$ on the Poisson manifold $P$, i.e., 
$X_{H_{\xi}}(p) = \xi_P(p)$. 
\end{definition}
Thus, Lie algebra actions that admit momentum maps are Hamiltonian actions 
and the pairing of the momentum map at a point $p\in P$ with an element $\xi \in \mathfrak g$ 
defines a Hamiltonian function associated with the Hamiltonian vector field $\xi_P$ at that point $p$. 

\smallskip

A  momentum map  ${\mathcal M}\colon P \to \mathfrak g^*$  of a Lie algebra $\mathfrak g$ 
is \textit{infinitesimally equivariant} if for all $\xi,\,\eta \in \mathfrak g$ one has
$H_{[\xi , \eta ]} = \{ H_{\xi} , H_\eta\}$, which means that not only for any Lie algebra vector defines 
a Hamiltonian vector field on the manifold, but also the Lie algebra bracket of two such fields 
corresponds to the Poisson bracket of  their Hamiltonians. 
\end{subsection}
%
%

\begin{subsection}{Madelung transform is a momentum map} 
\label{sec:madelung-moment} 

We now show (following Fusca~\cite{Fu2017}) that the transformation $\bf M$ is a momentum map 
associated with the action \eqref{eq:infact}. 

First note that 
the vector space  $H^s(M,\CC) \subset L^2(M,\CC)$ of Sobolev wave functions on $M$
is naturally equipped with 
the symplectic (and hence Poisson) structure 
$\{{F},{G}\}(\psi) =\inner{ \nabla F ,-\ii \nabla G }_{L^2}=\inner{ \ud F ,J \ud G }_{L^2}$. 
This structure is related to the natural 
Hermitian inner product on $L^2(M,\CC)$:  $\inner{ f, g}_{L^2}:=\int_M f \bar g \,\vol$ and the complex structure of multiplication  by $\ii$.
Now define the Hamiltonian function 
$H_\xi \colon H^s(M,\CC) \to \RR^{}$ by 
$H_\xi(\psi) \coloneqq \langle {\bf M} (\psi) , \xi \rangle$.
\begin{theorem} \label{thm:madmoment}{\rm \cite{Fu2017}} 
For the Lie algebra $\mathfrak s = \mathfrak{X}(M)\ltimes H^s(M,\RR)$ 
its action \eqref{eq:infact} on the Poisson space $H^s(M,\CC) \subset  L^2(M,\mathbb{C})$ 
admits a momentum map. 
The inverse Madelung transformation ${\bf M} \colon H^s(M,\CC) \to \mathfrak s^*$ 
defined by \eqref{eq:mad} is, up to scaling by $4$, a momentum map associated with this Lie algebra action.
\end{theorem}
\begin{proof}
The Hamiltonian vector field for the function $H_{\xi}$ is 
$X_{H_{\xi}} = -\ii \,\ud H_\xi$ where the differential is obtained from 
\begin{equation*} 
\pair{\ud H_\xi(\psi),\phi}
= 
\Re \, \inner{ \ud H_\xi(\psi),\phi }_{L^2} 
= 
\frac{\ud}{\ud t}\Big|_{\epsilon=0}H_\xi(\psi + \epsilon\phi) 
\end{equation*} 
for any $\phi$ in $H^s(M,\CC)$. 
Let $\xi = (v,\alpha)$ be an element of $\mathfrak s = \mathfrak{X}(M)\ltimes H^s(M)$ 
whose pairing with $(m, \rho)\in \mathfrak s^*$ is 
$\langle (v,\alpha), (m, \rho)\rangle:=\int_{M} (\rho \cdot \alpha + m\cdot v )\, \vol$. 
We have
\begin{align*}
H_\xi(\psi) 
&= 
\int_{M}( {\bf M} (\psi)^\rho\cdot \alpha + {\bf M} (\psi)^m \cdot v )\,\vol 
\\
&= 
\int_{M}( \bar\psi  \psi \alpha + 2\Im (\iota_v\bar\psi \,  \ud\psi ))\, \vol 
\\
&= 
\Re  \int_{M}( \bar\psi  \psi \alpha - 2\ii \,\iota_v\bar\psi \,  \ud\psi )\, \vol. 
\end{align*} 
To find the variational derivative let $\phi \in H^s(M,\CC)$ be a test function. 
Then 
\begin{align*} 
\frac{d}{d\epsilon} H_\xi(\psi + \epsilon \phi)|_{\epsilon = 0}
&=  
\Re  \int_{M} \Big( \bar\psi \phi \alpha + \bar\phi \psi \alpha 
- 
2\ii \,\iota_v\bar\phi \,  \ud\psi - 2\ii \,\iota_v\bar\psi \,  \ud\phi \Big) \vol 
\\ 
&= 
\Re  \int_{M} \Big( 2 \bar\phi \psi \alpha - 2\ii \bar\phi \,\iota_v\,\ud \psi + 2\ii \phi \divv(\bar\psi v) \Big) \vol 
\\ 
&= 
\Re  \int_{M} \Big( 2 \bar\phi \psi \alpha + \overline{2\ii \phi \iota_v \ud\bar\psi} 
+ 
2\ii \phi \bar\psi\divv(v) + 2\ii \phi \iota_v \ud\bar\psi \Big) \vol 
\\ 
&= 
\Re  \int_{M} \Big( 2 \psi \alpha - 4\ii \,\iota_v\,\ud \psi - 2\ii \psi \divv( v) \Big) \bar\phi \, \vol 
\\ 
&= 
\Re\, \inner{ 2 \psi \alpha - 4\ii \,\iota_v\,\ud \psi - 2\ii \psi \divv( v),\phi }_{L^2}, 
\end{align*} 
so that $\ud H_{\xi}(\psi) = 2 \psi \alpha - 2\ii \psi \divv( v) - 4\ii  \,\iota_v\,\ud \psi$. 
This implies that 
\begin{align*} 
X_{H_{\xi}}(\psi) = -2\ii\alpha \psi -4\iota_v\,\ud \psi  - 2 \psi \divv{v} 
\end{align*} 
and comparing with \eqref{eq:infact} one obtains $X_{H_{\xi}}(\psi) = 4V_\xi(\psi)$. 
\end{proof}
Moreover, the Madelung transform turns out to be an infinitesimally equivariant
 momentum map, as was verified in \cite{Fu2017}. (Recall that its equivariance means morphism of the 
 Lie algebras: the Hamiltonian of the Lie  bracket of two  fields is the Poisson bracket of their Hamiltonians.) 
In particular, it follows that the Madelung transform is also a Poisson map taking the Poisson structure 
on $P$ (up to scaling by 4) to the Lie-Poisson structure on $\mathfrak g$, i.e., 
the map ${\bf M} \colon H^s(M,\CC) \to \mathfrak s^*$ is infinitesimally equivariant for the action 
on $H^s(M,\CC)$ of the semidirect product Lie algebra $\mathfrak s$.
This result is expected from the symplectomorphism result in \autoref{thm:madelung_symplectomorphism} 
since $T^*\Dens^s(M)$ is a coadjoint orbit in $\mathfrak s^*$ via $(\rho,[\theta])\mapsto (\rho\,\ud\theta,\rho)$.
\end{subsection}

\begin{subsection}{Multi-component Madelung transform as a momentum map}\label{sec:v-madelung}

There is a natural generalization of the above approach to the space of wave vector-functions
$\psi\in H^s(M,\CC^\ell)$, notably rank 1 spinors for which $\ell =2$.
One needs to define an action of the group $S_{(\ell)}=\Diff(M)\ltimes H^s(M)^\ell$ on the  subspace of smooth vector-functions.

\begin{definition}
	The semidirect product group $S_{(\ell)}=\Diff(M)\ltimes H^s(M)^\ell$ {\it acts on the  space} 
	$\mathbb PH^s(M,\CC^{\ell})$  as follows: if $(\varphi,a)\in S_{(\ell)}$ is a group element, where $\varphi$ is a diffeomorphism,  $\tilde a=(a_1,...,a_\ell)$ is a vector, and $\tilde\psi=(\psi_1,...,\psi_\ell)$ is a smooth wave vector-function, then
\begin{equation}\label{eq:m-grpact}
		(\varphi,a_1,...,a_\ell): \psi_k \mapsto (\varphi,\tilde a) \circ \psi_k := \sqrt{|\mathrm{Det}(D\varphi^{-1})|}\, e^{-ia_k/2} (\psi_k \circ \varphi^{-1})
\end{equation}
for $k=1,...,\ell$.
The corresponding Lie algebra is denoted $\mathfrak s_{(\ell)}$.
\end{definition}
%
\begin{definition}
The (inverse) {\it multicomponent Madelung transform} is the map 
${\bf M}^{(\ell)} \colon H^s(M,\CC^{\ell}) \to \mathfrak s_{(\ell)}^*$ 
defined by 
${\bf M}^{(\ell)}(\tilde \psi) = ( m  ,\tilde \rho )$, 
where 
$ 
m = 2 \sum_{k=1}^\ell \Im\,(\bar\psi_k \,\ud \psi_k)
$ 
and $\tilde \rho=(\rho_1,\ldots,\rho_\ell)$ with $\rho_k\coloneqq\bar\psi_k \psi_k$. 
\end{definition} 
Here, as before, we have 
$ 
m\otimes \vol\in \Omega^1(M)\otimes_{H^s(M)} \Omega^n(M) 
$ 
while for each $k=1, \dots \ell$ we have $\rho_k\in H^s(M)$, so that 
$( m,\tilde \rho ) \otimes \vol\in \mathfrak s_{(\ell)}^*$. 

Similarly, the space $H^s(M,\CC^{\ell})$ has a natural symplectic (and hence Poisson) structure 
and one can prove a multicomponent version of \autoref{thm:madmoment}. 
\begin{theorem}\label{thm:m-madmoment}
For the Lie algebra $\mathfrak s_{(\ell)} = \mathfrak{X}(M)\ltimes H^s(M)^\ell$, 
its action \eqref{eq:m-grpact} on the Poisson space $H^s(M,\CC^{\ell})$ admits a momentum map. 
The map ${\bf M}^{(\ell)} \colon H^s(M,\CC^{\ell}) \to \mathfrak s_{(\ell)}^*$  is a momentum map 
associated with this Lie algebra action. 
The Madelung transform is (up to scaling by $4$) a Poisson map taking the bracket 
on $H^s(M,\CC^{\ell})$ to the Lie-Poisson bracket on the dual $\mathfrak s_{(\ell)}^*$ 
of the semidirect product Lie algebra $\mathfrak s_{(\ell)}$.
\end{theorem}
\begin{remark} \label{rmk:SUsemidirect} 
More generally, for any subgroup $G$ of $\mathrm{U}(\ell)$ one has an action on $H^s(M,\CC^\ell)$
of the semidirect product $\widetilde S_{(\ell)}=\mathrm{Diff}(M)\ltimes H^s(M,G)$ 
of diffeomorphisms $\varphi$ with $G$-valued $H^s$-functions $A\in H^s(M,G)=H^s(M)\otimes G$ 
on $M$. It is given by 
$$ 
(\varphi, A): \tilde \psi \mapsto   \sqrt{|\mathrm{Det}(D\varphi^{-1})|}\, A (\tilde \psi \circ \varphi^{-1}). 
$$ 
In particular, if $\ell=2$ (or $\ell=4$) the subgroup $G=\SU(\ell)$ acts by rotation of spinors 
(this may have some relevance for hydrodynamic formulations of the Pauli (or Dirac) equations). 
Note that the action of $\widetilde S_{(\ell)}$ preserves the Hermitian and symplectic structures 
on $H^s(M,\CC^\ell)$ and admits a momentum map.
\end{remark} 
\begin{remark} 
From the viewpoint of Hamiltonian dynamics specifying a larger $\ell$ 
(and considering the corresponding semi-direct product groups $S_{(\ell)}$) 
corresponds to ``exploring a larger chunk'' of the phase space 
$T^*\Diff(M)\simeq \Diff(M)\times \Xcal(M)^*$ (cf.\ next section). 
Indeed, for $\ell=1$ the corresponding equations on $T^*\Dens^s(M)$ only allows for momenta 
of the form $m=\rho\,\ud\theta$ (corresponding to potential-type solutions of the barotropic Euler equations). 
By choosing $\ell>0$ we allow for momenta of the form $m=\sum_{k=1}^\ell \rho_k\,\ud\theta_k$ 
thus filling out a larger portion of $\Xcal(M)^*$.
\end{remark}
%


\subsection{Example: general compressible fluids}

For general compressible (non\-baro\-tropic) inviscid fluids the equation of state describes the pressure 
as a function $P(\rho, \sigma)$ of both density $\rho$ and entropy $\sigma$. 
Thus, the corresponding equations of motion include the evolution of all three quantities: 
the velocity $v$ of the fluid, its density $\rho$ and the entropy $\sigma$. 
\begin{equation*} 
\left\{
 \begin{array}{l}
			\dot v + \nabla_v v + \frac{1}{\rho}\nabla P(\rho, \sigma) = 0 \,, \\
			\dot\rho + \divv(\rho v) = 0 \,,\\
                         \dot \sigma +\LieD_v \sigma=0\,.\\
\end{array} \right. 
\end{equation*} 
In the case the entropy is constant or the pressure is independent of $\sigma$ 
this system describes a barotropic flow, see equations \eqref{eq:barotropic2}. 
Note that, while the density evolves as an $n$-form, the entropy evolves as a function. 
However, according to the continuity equation, passing to the entropy density 
$\varsigma=\sigma\rho$ one can regard the corresponding group as the semidirect product
$S_{(2)}=\Diff(M)\ltimes (H^s(M)\oplus H^s(M))$, 
which leads to a Hamiltonian picture on the dual $\mathfrak s_{(2)}^*$. 
By applying the multicomponent Madelung transform ${\bf M}^{(2)}$ one can rewrite and interpret this system 
on the space of rank-1 spinors $H^s(M,\CC^2)$. 
Indeed, the evolution of the momentum $m\otimes\vol=v^\flat\otimes\vol$ is 
\begin{equation*}
\dot m \otimes\vol + \LieD_v(m\otimes\vol) 
+ 
\ud \frac{\delta U}{\delta \rho} \otimes \rho\,\vol 
+ 
\ud \frac{\delta U}{\delta \varsigma} \otimes \varsigma\,\vol 
= 0. 
\end{equation*} 
Observe that an invariant subset of solutions is given by those with momenta 
$m=\rho\,\ud\theta+\varsigma\,\ud\tau$, 
where $\theta,\tau \in H^s(M)/\RR$. 
They can be regarded as analogs of potential solutions of the barotropic fluid equations. 
We thereby obtain a canonical set of equations on $T^*H^s(M)^2$ given by 
\begin{equation*} 
\left\{
\begin{array}{lcl}
	\dot\rho = \frac{\delta H}{\delta\theta} & 	\dot\varsigma = \frac{\delta H}{\delta\tau} \\
	\dot\theta = -\frac{\delta H}{\delta\rho} &\quad	\dot\tau = -\frac{\delta H}{\delta\varsigma} 
\end{array} \right. 
\end{equation*} 
%
for a Hamiltonian of the form 
\begin{equation*}
H(\rho,\varsigma,\theta,\tau) 
= 
\frac{1}{2}\int_M (\abs{\nabla\theta}^2\rho + \abs{\nabla\tau}^2\varsigma)\vol + U(\rho,\varsigma). 
\end{equation*} 
Using the multicomponent Madelung transform 
\begin{equation}
(\rho,\varsigma,\theta,\tau) \mapsto \Big( \sqrt{\rho}\ee^{\ii\theta/2}, \sqrt{\varsigma}\ee^{\ii\tau/2} \Big), 
\end{equation} 
and \autoref{thm:m-madmoment} 
this gives (up to scaling by 4) a Hamiltonian system for the spinor 
$\tilde\psi = (\psi_1,\psi_2)\in H^s(M,\CC^2)$ 
with the Hamiltonian given by 
\begin{equation*} 
H(\tilde\psi) = \frac{1}{2} \big\| \nabla\tilde\psi \big\|_{L^2}^2 + W\big( \abs{\psi_1}^2,\abs{\psi_2}^2 \big). 
\end{equation*} 
for the potential 
\begin{equation*}
W(\rho,\varsigma) 
= 
- \frac{1}{2} \big\| \nabla \sqrt{\rho} \big\|_{L^2}^{2} 
- 
\frac{1}{2} \big\| \nabla \sqrt{\varsigma} \big\|_{L^2}^{2} 
+ 
\frac{1}{4}U(\rho,\varsigma), 
\end{equation*} 
where the functional $U$ is related to the pressure function $P(\rho,\sigma)$ 
of the compressible Euler equation. 
The corresponding Schr\"{o}dinger equation reads 
\begin{align*} 
\ii\dot{\tilde\psi} 
= 
-\Delta \tilde\psi 
+ 
\begin{pmatrix}\frac{\delta W}{\delta\rho}&0 \\ 0&\frac{\delta W}{\delta\varsigma}\end{pmatrix}\tilde\psi. 
\end{align*} 

Conversely, one can work backwards to obtain a fluid formulation of various quantum-mechanical
spin Hamiltonians, such as the Pauli equations for spin $1/2$ particles of a given charge.





\end{subsection}

\subsection{Geometry of semi-direct product reduction} 
\label{sect:general-semi} 
In this section we present the geometric structure behind the semi-direct product reduction 
which reveals the origin of the Madelung transform as the moment map above. 

\newcommand{\subG}{N}
\newcommand{\algG}{\mathfrak{n}}

Let $\subG$ be a Lie subgroup of a Lie group $G$. 
Assume that $G$ acts from the left on a linear space $V$ (a left representation of $G$). 
The quotient space of left cosets $G/\subG$ is acted upon from the left by $G$. 
Assume now that $G/\subG$ can be embedded as an orbit in $V$ and 
let $\gamma\colon G/\subG \to V$ denote the embedding. 
Since the action of $G$ on $V$ induces a linear left dual action on $V^*$ 
we can construct the semi-direct product $S=G\ltimes V^*$. 
\begin{proposition} 
The quotient $T^*G/\subG$ is naturally embedded via a Poisson map in the Lie-Poisson space 
$\mathfrak{s}^*$ 
(the dual of the corresponding semi-direct product algebra). 
\end{proposition} 
\begin{proof} 
The Poisson embedding is given by 
\begin{equation}\label{eq:semi_direct_map} 
([g],m) \mapsto (m,\gamma([g])) 
\end{equation} 
where we use that 
$T^*G/\subG \simeq G/\subG \times \mathfrak{g}^*$ 
and 
$\mathfrak{s}^* \simeq \mathfrak{g}^*\times V$. 
Now, the action of $S$ on $\mathfrak{s}^*$ is 
\begin{equation} 
(g,a)\cdot (m,b) 
= 
\mathrm{Ad}^*_{(g,a)}(m,b) 
= 
\big( \mathrm{Ad}^*_{g}(m) - {\mathcal M}(a,b), g\cdot b \big), 
\end{equation} 
where $ {\mathcal M}\colon V^*\times V \to \mathfrak{g}^*$ is the momentum map 
associated with the cotangent lifted action of $G$ on $V^*$. 
The corresponding infinitesimal action of $\mathfrak{s}$ is
\begin{equation} 
\label{eq:inf_ham_action} 
(\xi,\dot a)\cdot (m,b) 
= 
\mathrm{ad}^*_{(\xi,\dot a)}(m,b) 
= 
\big( \mathrm{ad}^*_{\xi}(m) - {\mathcal M}(\dot a,b), \xi\cdot b \big). 
\end{equation} 
Since the second component is only acted upon by $g$ (or $\xi$), but not $a$ (or $\dot a$), 
it follows from the embedding of $G/\subG$ as an orbit in $V$ that we have a natural Poisson action 
of $S$ (or $\mathfrak{s}$) on $T^*G/\subG$ via the Poisson embedding~\eqref{eq:semi_direct_map}. 
Notice that the momentum map of $S$ (or $\mathfrak{s}$) acting on $\mathfrak{s}^*$ is the identity: 
this follows since the Hamiltonian vector field on $\mathfrak{s}^*$ 
for $H(m,b) = \pair{m,\xi} + \pair{b,\dot a}$ is given by~\eqref{eq:inf_ham_action}. 
\end{proof} 

We now return to the standard symplectic reduction (without semi-direct products). 
The dual $\algG^*$ of the subalgebra $\algG\subset \mathfrak{g}$ is naturally identified with 
affine cosets of $\mathfrak{g}^*$ such that 
\begin{equation}
m \in [m_0] \iff \pair{m-m_0,\xi} = 0 \qquad \forall\, \xi\in\algG. 
\end{equation} 
The momentum map of the subgroup $\subG$ acting on $\mathfrak{g}^*$ by $\Ad^*$ 
is then given by 
$m \mapsto [m],$ 
since the momentum map of $G$ acting on $\mathfrak{g}^*$ is the identity. 
If $\pair{m,\algG} = 0$, i.e., $m\in (\mathfrak{g}/\algG)^*$, then $m\in [0]$ is in the zero momentum coset. 
Since we also have $T^*(G/\subG) \simeq G/\subG\times (\mathfrak{g}/\algG)^*$ 
this gives us an embedding as a symplectic leaf in $T^*G/\subG\simeq G/\subG\times \mathfrak{g}^*$. 
The restriction to this leaf is called \emph{zero-momentum symplectic reduction}. 

Turning to the semi-direct product reduction, we now have Poisson embeddings of 
$T^*(G/\subG)$ in $T^*G/\subG$ and of $T^*G/\subG$ in $\mathfrak{s}^*$. 
The combined embedding of $T^*(G/\subG)$ as a symplectic leaf in $\mathfrak{s}^*$ 
is given by the map 
\begin{equation}\label{eq:embed} 
([g],a)\mapsto ( {\mathcal M}(a,\gamma([g])),\gamma([g])) 
\end{equation} 
This implies that we have a Hamiltonian action of $S$ (or $\mathfrak{s}$) 
on the zero-momentum symplectic leaf $T^*(G/\subG)$ sitting inside $T^*G/\subG$, 
which in turn sits inside $\mathfrak{s}^*$.

Since $S$ has a natural symplectic action on $\mathfrak{s}^*$ and since $G/\subG$ 
is an orbit in $V \simeq V^{**}$, we have, by restriction, a natural action of $S$ on $T^*G/\subG$. 
Furthermore, since the momentum map associated with $S$ acting on $\mathfrak{s}^*$ is the identity, 
the Poisson embedding map \eqref{eq:semi_direct_map} is the momentum map for $S$ 
acting on $T^*G/\subG$. 
Thus, the momentum map of $S$ acting  on $T^*(G/\subG)$ is given by~\eqref{eq:embed}. 

The above considerations are summarized in the following theorem.
\begin{theorem} 
The inverse of the Madelung transform viewed as a momentum map 
(Section \ref{sec:madelung-moment}) can be regarded as the semi-direct product reduction 
and a Poisson embedding $T^*(G/\subG) \to \mathfrak{s}^*$ 
as described above for the groups $G=\Diff(M)$ and $\subG=\Diffvol(M)$.
\end{theorem}


\medskip

\appendix
\section{The functional-analytic setting} \label{sec:analysis} 

The infinite-dimensional geometric constructions in this paper can be rigorously carried out 
in any reasonable function space setting in which the topology is at least as strong as $C^1$, 
satisfies the functorial axioms of Palais~\cite{Pa1968} and admits a Hodge decomposition. 
The choice of the Sobolev spaces is very convenient for the purposes of this paper 
because many of the technical details which were used (explicitly or implicitly) in the proofs 
can be readily traced in the literature. 
We briefly review the main points below. 
 
As introduced in the main body of the paper the notation $\Diff^s(M)$ stands for the completion of 
the group of smooth $C^\infty$ diffeomorphisms of an $n$-dimensional compact Riemannian manifold $M$ 
with respect to the $H^s$ topology where $s>n/2+1$. This puts the Sobolev lemma at our disposal and 
thus equipped $\Diff^s(M)$ becomes a smooth Hilbert manifold whose tangent space at the identity 
$T_e\Diff^s(M)$ consists of all $H^s$ vector fields on $M$, see e.g., \cite{EbMa1970}, Section 2. 

Using the implicit function theorem the subgroup 
$\Diff^{s}_\vol(M) = \{ \eta \in \Diff^{s}(M) : \eta^\ast \vol = \vol \}$ 
consisting of those diffeomorphisms that preserve the Riemannian volume form $\vol$ 
can then be shown to inherit the structure of a smooth Hilbert submanifold 
with $T_e\Diff^s_\vol(M) = \{ v \in T_e\Diff^s : \mathrm{div}\, v = 0 \}$, 
cf. e.g., \cite{EbMa1970}, Sections 4 and 8. 

Standard results on compositions and products of Sobolev functions ensure that 
both $\Diff^{s}$ and $\Diff^{s}_\vol$ 
are topological groups with right translations $\xi \to \xi\circ\eta$ (resp., left translations $\xi \to \eta\circ\xi$ 
and inversions $\xi \to \xi^{-1}$) being smooth (resp., continuous) as maps in the $H^s$ topology, 
cf. \cite{Pa1968}; Chapters 4 and 9. 
Furthermore, the natural projection 
$$ 
\pi : \Diff^{s+1}(M) \to \Diff^{s+1}(M)/\Diff^{s+1}_\vol(M) 
\simeq 
\mathrm{Dens}^s(M) 
$$ 
given by $\eta \to \pi(\eta) = \eta^\ast \vol$ 
extends to a smooth submersion between $\Diff^{s+1}(M)$ and the space of right cosets 
which can be identified with the space of probability densities on $M$ of Sobolev class $H^{s}$ 
(cf.\ Section \ref{sec:madelung} above). 
More technical details, as well as proofs of all these facts, can be found in \cite{EbMa1970, Pa1968} and their bibliographies.

\section{A comment on rescaling constants} 

First, recall that for any $(\dot\rho,[\dot\theta]) \in T_{(\rho,[\theta])}T^*\Dens^s(M)$ 
we can pick a representative $\dot\theta\in[\dot\theta]$ such that 
$\int_M \dot\theta\rho\,\vol = 0$. 
The canonical symplectic structure on $T^\ast\Dens^s$ is then given by  
\begin{equation*} 
\Omega_{(\rho,[\theta])}\big((\dot\rho_1,\dot\theta_1),(\dot\rho_2,\dot\theta_2)\big) 
= 
\int_M  \Big( \dot\theta_2\dot\rho_1 - \dot\theta_1\dot\rho_2 \Big)\vol. 
\end{equation*} 
Furthermore, for any $\alpha > 0$ we have a complex structure 
\begin{equation*} 
J_{(\rho,[\theta])}(\dot\rho,\dot\theta) 
= 
\left( -\alpha\dot\theta\rho, \frac{1}{\alpha}\frac{\dot\rho}{\rho} \,\right). 
\end{equation*} 
Combining the two structures in a standard manner yields a K\"{a}hler metric on $T^\ast\Dens^s(M)$ 
\begin{align*} 
\MetF^*_{(\rho,[\theta])}\big((\dot\rho_1,\dot\theta_1),(\dot\rho_2,\dot\theta_2)\big) 
&= 
\Omega_{(\rho,[\theta])}\big((\dot\rho_1,\dot\theta_1), J_{(\rho,[\theta])}(\dot\rho_2,\dot\theta_2)\big) 
\\ 
&= 
\int_M \left( \frac{1}{\alpha}\frac{\dot\rho_1\dot\rho_2}{\rho} 
+ 
\alpha \dot\theta_1\dot{\theta}_2\rho \right)\vol. 
\end{align*} 
%


Next, we turn to the Madelung transform which, for a fixed constant $\gamma \neq 0$, 
is
\begin{equation}\label{eq:mad_with_constants} 
\psi = \sqrt{\rho \ee^{\ii\theta/\gamma}} 
\end{equation} 
and whose derivative is
\begin{equation}\label{eq:mad_deriv} 
\dot\psi 
= 
\frac{\psi}{2}\left( \frac{\dot\rho}{\rho} + \frac{\ii}{\gamma}\dot\theta \right). 
\end{equation} 
%
Given $\beta > 0$ the (scaled) Fubini-Study Hermitian structure on $\mathbb PC^\infty(M,\CC)$ 
is given by 
\begin{equation} \label{eq:H-beta} 
\FS_{\psi}(\dot\psi_1,\dot\psi_2) 
= 
\beta\frac{\inner{ \dot\psi_1,\dot\psi_2}_{L^2}}{\inner{\psi,\psi }_{L^2}} 
- 
\beta\frac{\inner{\dot\psi_1,\psi}_{L^2}\inner{\psi,\dot\psi_2}_{L^2}}{\inner{\psi,\psi}_{L^2}^2} 
\end{equation} 
where $\inner{\phi,\psi}_{L^2}= \int_M\phi \overline\psi \vol$. 
If $\dot\psi_1,\dot\psi_2$ are of the form \eqref{eq:mad_deriv}, then it follows from
$\int_M \rho\,\vol = 1$ and $\int_M\dot\rho_i\vol = \int_M\dot\theta_i\rho\,\vol = 0$ 
that $\norm{\psi}_{L^2} = 1$ and $\inner{\dot\psi_i,\psi}_{L^2} = 0$.
In this case 
\begin{align*} 
\FS_{\psi}(\dot\psi_1,\dot\psi_2) 
&= 
\beta\inner{\dot\psi_1,\dot\psi_2}_{L^2} 
= 
\frac{\beta}{4}\int_M \abs{\psi}^2 \left(\frac{\dot\rho_1}{\rho} 
+ 
\frac{\ii}{\gamma}\dot\theta_1 \right)\left(\frac{\dot\rho_1}{\rho} - \frac{\ii}{\gamma}\dot\theta_1 \right) \vol 
\\
&= 
\frac{\beta}{4\gamma} \int_M \left\{ 
\left(\gamma\frac{\dot\rho_1}{\rho}\dot\rho_2 
+ 
\frac{\dot{\theta}_1\dot{\theta}_2}{\gamma}\rho \right) 
- 
\ii \big( \dot\theta_2\dot{\rho}_1 - \dot\theta_1\dot{\rho}_2 \big) \right\} \vol
\end{align*} 
and the associated symplectic structure is 
\begin{align*} 
\hat\Omega_{\psi}(\dot\psi_1,\dot\psi_2) 
&= 
\Re\, \FS_{\psi}(\ii\dot\psi_1,\dot\psi_2) 
= 
-\Im\, \FS_{\psi}(\dot\psi_1,\dot\psi_2) 
\\ 
&= 
\frac{\beta}{4\gamma}\int_M \big( \dot\theta_2\dot\rho_1 - \dot\theta_1\dot\rho_2 \big)\vol 
= 
\frac{\beta}{4\gamma}\Omega_{(\rho,[\theta])}(\dot\rho_1,\dot\theta_1,\dot\rho_2,\dot\theta_2). 
\end{align*} 
Thus, the Madelung transform as defined by \eqref{eq:mad_with_constants} is a symplectomorphism 
up to a rescaling by the constant $\beta/4\gamma$. 

Similarly, the Riemannian metric associated with \eqref{eq:H-beta} is 
\begin{align*} 
\hat \MetF_{\psi}^*(\dot\psi_1,\dot\psi_2) 
&= 
\Re\, \FS_{\psi}(\dot\psi_1,\dot\psi_2)  
\\ 
&= 
\frac{\beta}{4\gamma}\int_M \left(\gamma\frac{\dot\rho_1 \dot\rho_2}{\rho} 
+ 
\frac{\dot{\theta}_1\dot{\theta}_2}{\gamma}\rho \right) \vol.
\end{align*} 
Thus, to make the Madelung transform defined by \eqref{eq:mad_with_constants} an isometry 
(up to rescaling by $\beta/4\gamma$) we require that $\alpha=\gamma$. 
If, in addition, it is to be a K\"{a}hler morphism, then we also require $\beta = 4\gamma$. 
Note that in this paper we set $\gamma=1$ while in his original work Madelung used $\gamma = \hbar/2$ 
(as did \citet{Re2012}).

\IfFileExists{/Users/moklas/Documents/Papers/References.bib}{
\bibliographystyle{amsplainnat}
\bibliography{/Users/moklas/Documents/Papers/References} 
}{
\IfFileExists{madelung.bbl}{
\def\cprime{$'$}

}{}
}
 
\end{document}